\documentclass[preprint,12pt]{elsarticle}

\usepackage{amsmath,amssymb,amsthm,mathtools}
\usepackage[utf8]{inputenc}
\usepackage{lmodern}
\usepackage[english]{babel}
\usepackage{microtype}
\usepackage{hyperref}
\hypersetup{hidelinks}
\usepackage{xcolor}

\providecommand{\Subset}{\mathrel{\subset\!\!\subset}}

\allowdisplaybreaks

\newtheoremstyle{ptheorem}{1em}{0em}{\itshape}{}{\bfseries}{.}{.5em}{\thmname{#1}\thmnumber{
		#2}\thmnote{ (\hspace{-1sp}{#3})}}
\theoremstyle{ptheorem}
\newtheorem{thm}{Theorem}[section]
\newtheorem{pro}[thm]{Proposition}
\newtheorem{lem}[thm]{Lemma}
\newtheorem{cor}[thm]{Corollary}

\newtheoremstyle{hdef}{1em}{0em}{}{}{\bfseries}{.}{.5em}{\thmname{#1}\thmnumber{
		#2}\thmnote{ (\hspace{-.01pt}{#3})}}
\theoremstyle{hdef}
\newtheorem{dfn}[thm]{Definition}
\newtheorem{rem}[thm]{Remark}
\newtheorem{exa}[thm]{Example}

\numberwithin{equation}{section}

\begin{document}

\selectlanguage{english}

\begin{frontmatter}

\title{Aubin--Lions Compactness for Stieltjes--Bochner Evolution Graphs}

\author[usc,citmaga]{Francisco J. Fern\'andez\corref{cor1}\fnref{orcid1}}
\ead{fjavier.fernandez@usc.es}

\cortext[cor1]{Corresponding author.}
\fntext[orcid1]{ORCID: \href{https://orcid.org/0000-0002-0667-2639}{0000-0002-0667-2639}.}

\affiliation[usc]{
  organization={Departamento de Estat\'istica, An\'alise Matem\'atica e Optimizaci\'on,
    Facultade de Matem\'aticas, Universidade de Santiago de Compostela},
  addressline={Campus Vida},
  city={Santiago de Compostela},
  postcode={15782},
  country={Spain}}

\affiliation[citmaga]{
  organization={Galician Center for Mathematical Research and Technology (CITMAga)},
  city={Santiago de Compostela},
  postcode={15782},
  country={Spain}}

\begin{abstract}
Let $g$ be a nondecreasing left-continuous function and let $\mu_g$ be its
Lebesgue--Stieltjes measure. We establish a Banach-valued fundamental theorem
of calculus and an Aubin--Lions compactness theorem for evolution measured by
$\mu_g$, allowing absolutely continuous, singular continuous and atomic
components. If the range space has the Radon--Nikod\'ym property, a curve is
$g$-absolutely continuous if and only if it is an indefinite Bochner
integral; its strong $g$-derivative is the Bochner density and the variation
measure has density equal to its norm. A terminal atom may make the derivative
invisible from the Bochner state class, so the natural evolution space is a
graph of state--derivative pairs. For
$B_0\Subset B\hookrightarrow B_1$, with $B_0$ and $B_1$ reflexive and
$1<p_0,p_1<\infty$, the state projection is a compact linear operator into
$L_g^{p_0}([a,b);B)$. The proof combines bounded evaluation at atoms, local
$\mu_g$-averages at nonatomic points and Ehrling's inequality. The result
recovers the classical theorem and yields compactness principles for bounded
time scales, weighted sequences and mixed Stieltjes measures.
\end{abstract}


\gdef\useauthors{Francisco J. Fernandez;}

\end{frontmatter}

\section{Introduction}\label{sec:introduction}

The Aubin--Lions lemma is a basic compactness tool for evolution equations.
For a Banach triple
\begin{displaymath}
 B_0\Subset B\hookrightarrow B_1,
\end{displaymath}
it converts spatial control in $L^{p_0}(a,b;B_0)$ and temporal control in
$L^{p_1}(a,b;B_1)$ into strong compactness in $L^{p_0}(a,b;B)$. We refer to
\cite{Aubin1963,Lions1969} for the original results and to \cite{Simon1987}
for a systematic treatment. Related compactness mechanisms have been
developed for integral equicontinuity and tightness \cite{RossiSavare2003},
evolving Banach spaces \cite{AlphonseEtAl2023}, nonlocal fractional
derivatives \cite{LiLiu2018}, and mesh-dependent discrete reconstructions
\cite{Gomez2026}. The present paper considers a different extension: time is
measured by a fixed Lebesgue--Stieltjes measure that may contain continuous,
singular and atomic components simultaneously.

Let $g:\mathbb R\to\mathbb R$ be nondecreasing and left-continuous, and let
$\mu_g$ be determined by
\begin{displaymath}
 \mu_g([c,d))=g(d)-g(c),
 \qquad c<d.
\end{displaymath}
A right jump of $g$ produces an atom of $\mu_g$, a plateau has zero
$\mu_g$-mass, and the continuous part of $\mu_g$ may be singular with respect
to Lebesgue measure. Thus a single derivator can encode ordinary evolution,
isolated events and inactive intervals. Stieltjes differential equations use
this observation to unify continuous and impulsive dynamics; see
\cite{POUSO2015,POUSO2017}. The relation with time-scale calculus is also
well established. The basic theory was introduced by Hilger and developed in
\cite{Hilger1990,BohnerPeterson2001}; the canonical ceiling map was used by
Slav\'\i k to represent dynamic equations as generalized ordinary
differential equations \cite[Lemma~4]{Slavik2012}, while L\'opez Pouso and
Rodr\'\i guez identified the Hilger delta derivative with a Stieltjes
derivative of the corresponding extension
\cite[Theorem~3.1]{POUSO2015}. Measure-theoretic descriptions of the Hilger
measure and delta integral are given in
\cite{Guseinov2003,CabadaVivero2006}.

Banach-valued Stieltjes--Bochner spaces and differentiation of indefinite
Bochner integrals were developed in \cite{Fran2020}, and scalar compactness
criteria were obtained in \cite{MR4707791}. These results do not provide the
three-space compactness principle required in weak evolution problems. The
aim here is to establish such a principle directly on the finite measure
space $([a,b),\mu_g)$, without assuming that $g$ is invertible, strictly
increasing or absolutely continuous.

Two issues must first be resolved. The first concerns the converse
fundamental theorem of calculus. If
\begin{displaymath}
 u(t)=u(a)+\int_{[a,t)}v(s)\,d\mu_g(s),
\end{displaymath}
then the differentiation theorem from \cite[Theorem~2.9]{Fran2020} gives
$u'_g=v$ almost everywhere. The converse is genuinely vector-valued:
differentiating all scalarizations $\langle y',u\rangle$ gives at most weak
information and does not produce a norm derivative or a strongly measurable
Bochner density. We therefore associate with $u$ a vector measure. When the
range space has the Radon--Nikod\'ym property, its defining vector-measure
representation yields a Bochner density and gives the desired integral
formula. The argument also identifies the total variation measure with the
measure whose density is $\|u'_g\|$. Since every reflexive Banach space has
the Radon--Nikod\'ym property, this hypothesis is natural for the compactness
applications considered below \cite[Ch.~III, Sects.~1--2]{DiestelUhl1977}.

The second issue is specific to derivators with a terminal atom. If
$s<b$ is an atom and $\mu_g((s,b))=0$, the Bochner state class records the
left-continuous value at $s$, whereas the derivative at the atom determines
the post-jump value. Distinct derivative classes may therefore correspond to
the same state class. Example~\ref{ex:terminal-atom-nonuniqueness} makes this
ambiguity explicit. The appropriate evolution space is consequently the
graph of the Stieltjes differentiation relation: its elements are admissible
pairs $(u,z)$ consisting of a state and its derivative. Here the word
``graph'' is used in the operator-theoretic sense, rather than for the
geometric graph of a single function.

For reflexive spaces $B_0$ and $B_1$, exponents $1<p_0,p_1<\infty$, and
\begin{displaymath}
 B_0\Subset B\hookrightarrow B_1,
\end{displaymath}
our main result states that the state projection
\begin{equation}\label{eq:intro-main-projection}
 \pi_0:\mathbb W_g^{p_0,p_1}([a,b];B_0,B_1)
 \longrightarrow L_g^{p_0}([a,b);B),
 \qquad \pi_0(u,z)=u,
\end{equation}
is compact. The proof works directly with $\mu_g$. At atoms, bounded point
evaluation and the compact embedding $B_0\Subset B_1$ give pointwise strong
compactness. At nonatomic points, local averages over shrinking half-open
intervals are compared with the canonical representative through
\begin{displaymath}
 \|\widetilde u(t)-\widetilde u(s)\|_{B_1}
 \leq
 \|z\|_{L_g^{p_1}(B_1)}
 \mu_g([s,t))^{1-1/p_1},
 \qquad s<t.
\end{displaymath}
This yields strong convergence in $L_g^{p_0}(B_1)$; Ehrling's inequality then
upgrades it to strong convergence in $L_g^{p_0}(B)$.

The theorem contains several familiar settings. For $g(t)=t$, it reduces to
the classical Aubin--Lions theorem. The canonical derivator of a bounded time
scale gives a delta-calculus compactness result. Purely atomic measures yield
weighted sequence compactness, while measures with Lebesgue, Cantor and
atomic components provide mixed continuous--singular--discrete models. The
contribution is therefore threefold: a vector-valued Stieltjes fundamental
theorem under a range-space condition that is sharp uniformly over the class
of derivators considered here (because that class contains the classical clock
$g(t)=t$), a well-posed evolution graph that retains terminal impulses, and an
Aubin--Lions theorem proved without reparametrizing the derivator.

From a functional-analytic viewpoint, the paper identifies the correct
closed graph on which Stieltjes differentiation acts when the state is only
specified as a Bochner equivalence class.  The Radon--Nikod\'ym property
governs the representation of its absolutely continuous vector measures,
whereas compactness of the state projection is obtained from weak compactness
in the outer Bochner spaces, compactness of the spatial embedding, and a
measure-adapted temporal modulus.  Thus the role of the clock and the role of
the Banach-space geometry remain explicitly separated.

Section~\ref{sec:preliminaries} fixes the conventions and collects the
Stieltjes--Bochner tools. Section~\ref{sec:fundamental-theorem} proves the
vector-valued fundamental theorem. Section~\ref{sec:aubin-lions} introduces
the evolution graph and establishes the compactness theorem.
Section~\ref{sec:consequences} derives convergence consequences, the
time-scale formulation, atomic and mixed examples, and the endpoint
limitations. Function spaces are defined on $([a,b),\mu_g)$, whereas
canonical representatives are defined on $[a,b]$ so that right jumps and
terminal values remain visible.

\section{Preliminaries for Stieltjes--Bochner analysis}\label{sec:preliminaries}

\subsection{Elementary Stieltjes calculus}

Throughout the paper, $a,b\in\mathbb R$ are fixed and $a<b$.
Let $g:\mathbb R\to\mathbb R$ be nondecreasing and left-continuous.  As in
\cite{POUSO2015}, we call $g$ a \emph{derivator}.  The Lebesgue--Stieltjes
measure generated by $g$ is denoted by $\mu_g$ and is fixed by the convention
\begin{displaymath}
 \mu_g([c,d))=g(d)-g(c),
 \qquad c<d.
\end{displaymath}
This convention and the corresponding measure construction are recalled in
\cite[Sect.~2]{POUSO2015} and \cite{CabadaFernandez2026}.  Measurability and almost-everywhere statements
with the prefix $g$ always refer to $\mu_g$.  If $E$ is $g$-measurable, we
write $L_g^p(E)=L^p(E,\mu_g)$, $1\leq p\leq\infty$.

Only the restriction of $\mu_g$ to $[a,b)$ enters the function spaces below.
It is a finite measure, and
\begin{equation}\label{eq:total-clock-mass}
 M_g:=\mu_g([a,b))=g(b)-g(a)<\infty.
\end{equation}
The half-open interval is not a cosmetic choice: it is the convention
compatible with a left-continuous derivator and its right jumps.

We shall use the following two subsets of the real line:
\begin{align*}
	C_g&=\bigl\{t\in\mathbb R:
	  g\text{ is constant on }(t-\varepsilon,t+\varepsilon)
	  \text{ for some }\varepsilon>0\bigr\},\\
	D_g&=\{t\in\mathbb R:\Delta^+g(t)>0\}.
\end{align*}
Here $\Delta^+g(t)=g(t^+)-g(t)$ and $g(t^+)$ is the right-hand limit of
$g$ at $t$.  Clearly $C_g\cap D_g=\varnothing$.  The set $C_g$ is open
\cite{POUSO2015}; hence its connected components form an at most countable
family of pairwise disjoint, possibly unbounded, open intervals.  We write
\begin{equation}\label{Cgdisj}
	C_g=\bigcup_{n\in\Lambda} I_n,
\end{equation}
where $\Lambda\subset\mathbb N$.  If $I_n=(a_n,b_n)$, its endpoints are
understood in the extended real line; only finite endpoints will be used
below.  This convention is relevant, for example, when a derivator is
extended constantly outside a compact interval.

For later use, we record the precise relation between right jumps and atoms.
If $t\in\mathbb R$, continuity from above applied to $[t,t+1/n)$ gives
\begin{equation}\label{eq:atom-mass-clock}
 \mu_g(\{t\})
 =\lim_{n\to\infty}\mu_g([t,t+1/n))
 =g(t^+)-g(t)=\Delta^+g(t).
\end{equation}
Consequently, the atoms of $\mu_g$ are precisely the points of $D_g$.
Moreover, $D_g\cap[a,b)$ is at most countable.  Indeed, for every
$k\in\mathbb N$ the set of atoms with mass at least $1/k$ is finite, since
their masses add up to at most $M_g$; taking the union over $k$ proves the
claim.

With this notation, we introduce the sets of nonatomic left and right
endpoints of the connected components of $C_g$:
\begin{equation}\label{eq:Ng-endpoints}
	\begin{aligned}
		N_g^-&:=\{a_n\in\mathbb R:n\in\Lambda\}\setminus D_g,\\
		N_g^+&:=\{b_n\in\mathbb R:n\in\Lambda\}\setminus D_g,\\
		N_g&:=N_g^-\cup N_g^+.
	\end{aligned}
\end{equation}
	The superscript records the side used in the pointwise definition of the
	Stieltjes derivative: the left at points of $N_g^-$ and the right at points
	of $N_g^+$.  Endpoints that also belong to $D_g$ are excluded because the
	atomic convention applies there.

Since $\Lambda$ is at most countable, $N_g$ is at most countable. Moreover,
$N_g\cap D_g=\emptyset$ and hence
\begin{displaymath}
	\mu_g(\{t\})=\Delta^+g(t)=0,
	\qquad t\in N_g.
\end{displaymath}
It follows that
\begin{equation}\label{eq:Ng-mug-null}
	\mu_g(N_g)=0.
\end{equation}
Likewise, $g$ is constant on each component $I_n$, so
$\mu_g(I_n)=0$. By~\eqref{Cgdisj} and countable additivity,
\begin{equation}\label{eq:CgNg-mug-null}
	\mu_g(C_g)=0,
	\qquad
	\mu_g(C_g\cup N_g)=0.
\end{equation}
Thus $C_g\cup N_g$ may be included in the exceptional set whenever a
statement is only required to hold $g$-almost everywhere.  The notation
$N_g^\pm$ is needed only for pointwise derivatives at the endpoints of a
plateau.

The following notion of continuity is adapted to the clock $g$; see
\cite[Definition~3.1]{POUSO2017}.
\begin{dfn}\label{dfncont}
	Let $A\subset\mathbb R$ and let $X$ be a normed space.  A function
	$f:A\to X$ is \emph{$g$-continuous at $t\in A$} if, for every
	$\varepsilon>0$, there is $\delta>0$ such that
	\begin{displaymath}
	 \|f(s)-f(t)\|_X<\varepsilon
	 \quad\text{whenever }s\in A\text{ and }|g(s)-g(t)|<\delta.
	\end{displaymath}
	We denote by $\mathcal C_g(A;X)$ the space of $g$-continuous functions.
	Its bounded subspace is denoted by $\mathcal{BC}_g(A;X)$ and is endowed
	with the supremum norm
	\begin{displaymath}
	 \|f\|_{\infty,X}=\sup_{t\in A}\|f(t)\|_X.
	\end{displaymath}
\end{dfn}

This notion coincides with continuity of
$f:(A,\tau_g)\to(X,d)$, where $\tau_g$ is the topology induced by the
pseudometric $\rho_g(s,t)=|g(s)-g(t)|$. In particular, a $g$-continuous map
is constant on every level set $g^{-1}(\{\gamma\})$, $\gamma\in g(A)$.

Only an almost-everywhere derivative is needed in this paper.  It is therefore
enough to define it outside the null set $C_g\cup N_g$.  This also avoids
imposing the endpoint assumptions needed when the derivative is defined at
every point; compare \cite[Definition~3.7]{FERNANDEZ2022126010}.

\begin{dfn}[Strong Stieltjes derivative]\label{def:strong-g-derivative}
	Let $Y$ be a Banach space and let $f:[a,b]\to Y$.  If
	$t\in[a,b)\setminus(C_g\cup N_g)$, set
	\begin{displaymath}
	 f'_g(t)=
	 \begin{cases}
	 \displaystyle
	 \lim_{s\to t}\frac{f(s)-f(t)}{g(s)-g(t)},
	 &t\notin D_g,\\[3mm]
	 \displaystyle
	 \frac{f(t^+)-f(t)}{\Delta^+g(t)},
	 &t\in D_g,
	 \end{cases}
	\end{displaymath}
	provided that the indicated limit exists in the norm of $Y$.  In the first
	line the limit is relative to $[a,b]$ and is taken through points for which
	$g(s)\ne g(t)$.  We call this the \emph{strong $g$-derivative} of $f$.
\end{dfn}

\begin{rem}\label{rem:strong-vector-g-derivative}
	The values assigned to $f'_g$ on $C_g\cup N_g$ do not affect the Bochner
	spaces below, by \eqref{eq:CgNg-mug-null}.  If a derivative is required at
	every point, one must use the one-sided conventions at $N_g^-$ and $N_g^+$
	and the plateau convention given in
	\cite[Definition~3.7]{FERNANDEZ2022126010}.  The same reference states the
	corresponding restrictions at the endpoints of $[a,b]$.
\end{rem}

The scalar fundamental theorem below uses the following notion, which will
also be used for Banach-valued maps.
\begin{dfn}[$g$-absolutely continuous function] \label{defgabs}
	Let $(Y,\|\cdot\|)$ be a normed space and let $u:[a,b]\to Y$. We say that
	$u$ is \emph{$g$-absolutely continuous} on $[a,b]$ if for every
	$\varepsilon>0$ there exists $\delta>0$ such that, for every finite
	collection $\{(a_i,b_i)\}_{i=1}^n\subset[a,b]$ of pairwise disjoint
	intervals satisfying
	\begin{displaymath}
		\sum_{i=1}^n\bigl(g(b_i)-g(a_i)\bigr)<\delta,
	\end{displaymath}
	one has
	\begin{displaymath}
		\sum_{i=1}^n\|u(b_i)-u(a_i)\|<\varepsilon.
	\end{displaymath}
	We denote the resulting class by $AC_g([a,b];Y)$, and omit $Y$ in the
	scalar case.
\end{dfn}

We shall use the following scalar fundamental theorem. It is stated here so
that the endpoint and almost-everywhere conventions needed later are explicit.

\begin{thm}[Scalar fundamental theorem; {\cite[Theorem~5.4]{POUSO2015}}]
\label{t5.4}
A function $F:[a,b]\to\mathbb R$ is $g$-absolutely continuous if and only if
all of the following conditions are satisfied:
\begin{enumerate}
\item $F'_g(x)$ exists for $g$-almost every $x\in[a,b)$;
\item $F'_g\in L_g^1([a,b))$;
\item for every $x\in[a,b]$,
\begin{displaymath}
 F(x)=F(a)+\int_{[a,x)}F'_g(s)\,d\mu_g(s).
\end{displaymath}
\end{enumerate}
\end{thm}

\subsection{Total variation and Stieltjes absolute continuity}

\begin{dfn}[Total variation]
Let $(Y,\|\cdot\|)$ be a normed space and let $u:[a,b]\to Y$.  Its total
variation on $[a,b]$ is
\begin{displaymath}
 \operatorname{Var}_{[a,b]}(u)
 :=\sup_P\sum_{i=1}^n\|u(t_i)-u(t_{i-1})\|,
\end{displaymath}
where the supremum is taken over all partitions
$P=\{a=t_0<t_1<\cdots<t_n=b\}$.  We write $u\in BV([a,b];Y)$ if this
quantity is finite.  If $I$ is an interval, $BV_{\mathrm{loc}}(I;Y)$ denotes
the class of functions that have bounded variation on every compact
subinterval of $I$.
\end{dfn}

\begin{rem}
Every function of bounded variation is bounded.  Indeed, for
$t_0,t\in[a,b]$, the triangle inequality and the definition of total
variation give
\begin{displaymath}
 \|u(t)\|
 \leq \|u(t_0)\|+\|u(t)-u(t_0)\|
 \leq \|u(t_0)\|+\operatorname{Var}_{[a,b]}(u).
\end{displaymath}
\end{rem}

We also use the standard additivity of variation on adjacent intervals.
For $a\leq c\leq b$,
\begin{displaymath}
\operatorname{Var}_{[a,b]}(u)
=
\operatorname{Var}_{[a,c]}(u)
+
\operatorname{Var}_{[c,b]}(u).
\end{displaymath}
This identity follows directly from the definition: one joins partitions of
the two subintervals for one inequality and inserts $c$ into an arbitrary
partition of $[a,b]$ for the other.

\begin{lem}[Absolute continuity implies $g$-continuity]
\label{lem:ACg-Cg}
Let $Y$ be a normed space.  Then
\begin{displaymath}
 AC_g([a,b];Y)\subset\mathcal C_g([a,b];Y).
\end{displaymath}
\end{lem}

\begin{proof}
Let $u\in AC_g([a,b];Y)$, fix $t\in[a,b]$, and let $\varepsilon>0$.
Choose $\delta>0$ from Definition~\ref{defgabs}.  If $s\in[a,b]$ and
$|g(s)-g(t)|<\delta$, put $c=\min\{s,t\}$ and $d=\max\{s,t\}$.  Since $g$
is nondecreasing,
\begin{displaymath}
 g(d)-g(c)=|g(s)-g(t)|<\delta.
\end{displaymath}
Applying Definition~\ref{defgabs} to the single interval $(c,d)$ gives
$\|u(s)-u(t)\|_Y<\varepsilon$.  Thus $u$ is $g$-continuous at $t$.  Notice
also that if $g(s)=g(t)$, the same argument works for every
$\varepsilon>0$, and hence $u(s)=u(t)$.
\end{proof}

We shall also need the following bounded-variation property.

\begin{pro}\label{prop:ACg-BPV}
	Let $(Y,\|\cdot\|)$ be a normed space. Then
	\begin{displaymath}
		AC_g([a,b];Y)
		\subset
		BV([a,b];Y).
	\end{displaymath}
	Consequently,
	$AC_g([a,b];Y)\subset\mathcal{BC}_g([a,b];Y)$.
\end{pro}

\begin{proof}
	Let $u\in AC_g([a,b];Y)$. We first consider the degenerate case
	$g(a)=g(b)$. Since $g$ is nondecreasing, it is constant on $[a,b]$. Given
	$s<t$ and any $\varepsilon>0$, the single interval $(s,t)$ has
	\begin{displaymath}
		g(t)-g(s)=0.
	\end{displaymath}
	Definition~\ref{defgabs} therefore yields
	$\|u(t)-u(s)\|<\varepsilon$. Since $\varepsilon$ is arbitrary,
	$u(t)=u(s)$; hence $u$ is constant and
	$\operatorname{Var}_{[a,b]}(u)=0$.

	Assume now that
	\begin{displaymath}
		G:=g(b)-g(a)>0.
	\end{displaymath}
	Apply Definition~\ref{defgabs} with $\varepsilon=1$, and let $\delta>0$
	be the corresponding number. Choose $m\in\mathbb N$ such that
	\begin{displaymath}
		h:=\frac{G}{m}<\delta,
	\qquad
		y_k:=g(a)+kh,
	\quad k=0,\ldots,m.
	\end{displaymath}
	We use a half-open partition of the whole range $[g(a),g(b)]$:
	\begin{displaymath}
		J_k:=[y_{k-1},y_k),
		\quad k=1,\ldots,m-1,
		\qquad
		J_m:=[y_{m-1},y_m].
	\end{displaymath}
	Define
	\begin{displaymath}
		E_k:=\{t\in[a,b]:g(t)\in J_k\},
		\qquad k=1,\ldots,m.
	\end{displaymath}
	The sets $E_k$ are pairwise disjoint and their union is exactly $[a,b]$;
	the last range interval contains the previously missing endpoint $g(b)$.
	Moreover, since $g$ is nondecreasing, every nonempty $E_k$ is order-convex
	(that is, it contains every point lying between any two of its points):
	if $s,t\in E_k$ and $s<r<t$, then $r\in E_k$.

	Let
	\begin{displaymath}
		s_0<s_1<\cdots<s_q
	\end{displaymath}
	be arbitrary points of one fixed nonempty set $E_k$. The intervals
	$(s_{i-1},s_i)$ are pairwise disjoint and
	\begin{displaymath}
		\sum_{i=1}^q
		\bigl(g(s_i)-g(s_{i-1})\bigr)
		=
		g(s_q)-g(s_0)
		\leq h<\delta.
	\end{displaymath}
	Hence
	\begin{equation}\label{eq:variation-inside-range-cell}
		\sum_{i=1}^q
		\|u(s_i)-u(s_{i-1})\|<1.
	\end{equation}
	In particular, any two points $s,t\in E_k$ satisfy
	\begin{equation}\label{eq:oscillation-inside-range-cell}
		\|u(t)-u(s)\|<1.
	\end{equation}

	List the nonempty cells in their natural order as
	\begin{displaymath}
		E_{k_1},\ldots,E_{k_r},
		\qquad k_1<\cdots<k_r,
	\end{displaymath}
	and choose one point $c_j\in E_{k_j}$ for each $j$. Since this is a finite
	family,
	\begin{displaymath}
		M:=\max_{1\leq p,q\leq r}\|u(c_p)-u(c_q)\|<\infty.
	\end{displaymath}
	Let $P=\{t_0<\cdots<t_n\}$ be an arbitrary partition of $[a,b]$. The
	increments whose endpoints belong to the same cell have, by
	\eqref{eq:variation-inside-range-cell}, total sum less than $1$ for each
	nonempty cell, and therefore contribute at most $r$ altogether.

	The cell index along the ordered partition $P$ is nondecreasing. Thus there
	are at most $r-1$ increments whose endpoints belong to different cells. If
	$t_{i-1}\in E_{k_p}$ and $t_i\in E_{k_q}$ with $p<q$, then
	\eqref{eq:oscillation-inside-range-cell} and the triangle inequality give
	\begin{displaymath}
		\begin{aligned}
		\|u(t_i)-u(t_{i-1})\|
		&\leq
		\|u(t_i)-u(c_q)\|
		+
		\|u(c_q)-u(c_p)\|
		+
		\|u(c_p)-u(t_{i-1})\|
		\\
		&\leq M+2.
		\end{aligned}
	\end{displaymath}
	Consequently,
	\begin{displaymath}
		\sum_{i=1}^n\|u(t_i)-u(t_{i-1})\|
		\leq
		r+(r-1)(M+2).
	\end{displaymath}
	The right-hand side is independent of $P$. Taking the supremum over all
	partitions proves that
	$\operatorname{Var}_{[a,b]}(u)<\infty$.

	Finally, a function of bounded variation on a compact interval is bounded.
	Together with Lemma~\ref{lem:ACg-Cg}, this gives
	$u\in\mathcal{BC}_g([a,b];Y)$.
\end{proof}

Let $(X,\mathfrak M,\mu)$ be a measure space and let $Y$ be a Banach space.
A map $f:X\to Y$ is \emph{weakly measurable} if
$y'\!\circ f$ is measurable for every $y'\in Y'$. It is \emph{strongly
measurable} if it is the $\mu$-almost-everywhere pointwise limit of a sequence
of $Y$-valued simple functions. A strongly measurable map is
\emph{Bochner integrable} if $\|f(\cdot)\|_Y\in L^1(X,\mu)$. For a $g$-measurable set $X$ and $1\leq p<\infty$, we denote by
$L_g^p(X;Y)$ the space of equivalence classes, modulo
$\mu_g$-almost-everywhere equality, of strongly $\mu_g$-measurable maps
$f:X\to Y$ for which
\begin{displaymath}
 \|f\|_{L_g^p(X;Y)}
 :=
 \left(\int_X\|f(x)\|_Y^p\,d\mu_g(x)\right)^{1/p}
 <\infty.
\end{displaymath}
The space $L_g^\infty(X;Y)$ is defined analogously using the essential
supremum.  For $Y=\mathbb R$ this agrees with the scalar notation introduced
above.

\subsection{Bochner integration and compactness tools}

Two auxiliary results will be used repeatedly.  We state them in the precise
form required by the later arguments.

The next elementary approximation fact spells out one step in the
differentiation argument.  The pointwise bound allows the scalar
differentiation theorem to be applied to the approximation error without a
second truncation argument.

\begin{lem}[Dominated simple approximation]
\label{lem:dominated-simple-approximation}
Let $(X,\mathfrak M,\mu)$ be a measure space, let $V$ be a Banach space, and
let $f:X\to V$ be strongly measurable.  There are measurable simple
functions $f_n:X\to V$ such that
\begin{equation}\label{eq:dominated-simple-approximation}
 f_n(x)\longrightarrow f(x)
 \quad\text{and}\quad
 \|f_n(x)\|_V\leq2\|f(x)\|_V
\end{equation}
for $\mu$-almost every $x\in X$.  In particular, if $f$ is Bochner
integrable, then every $f_n$ and every scalar function $\|f-f_n\|_V$ is
integrable.
\end{lem}

\begin{proof}
Strong measurability provides simple functions $\varphi_n:X\to V$ and a
set $N\in\mathfrak M$ with $\mu(N)=0$ such that
$\varphi_n(x)\to f(x)$ for every $x\in X\setminus N$.  For each $n$, set
\begin{displaymath}
 E_n
 =
 \bigl\{x\in X\setminus N:
 f(x)\ne0\text{ and }
 \|\varphi_n(x)-f(x)\|_V\leq\|f(x)\|_V\bigr\}
\end{displaymath}
and define $f_n=\mathbf1_{E_n}\varphi_n$.  The set $E_n$ is measurable and
$f_n$ is still a simple function.  Fix $x\in X\setminus N$.  If
$f(x)\ne0$, then $x\in E_n$ for all sufficiently large $n$, and hence
$f_n(x)=\varphi_n(x)\to f(x)$.  If $f(x)=0$, then $f_n(x)=0$ for every
$n$.  Therefore $f_n(x)\to f(x)$ for every $x\in X\setminus N$.

On $E_n$, the triangle inequality gives
\begin{displaymath}
 \|f_n(x)\|_V
 \leq\|f(x)\|_V+\|\varphi_n(x)-f(x)\|_V
 \leq2\|f(x)\|_V,
\end{displaymath}
and the same inequality is immediate on
$(X\setminus N)\setminus E_n$.  Thus the pointwise bound holds
$\mu$-almost everywhere.  Finally,
$\|f-f_n\|_V\leq3\|f\|_V$ $\mu$-almost everywhere, which proves the
integrability assertion.
\end{proof}

\begin{thm}[Differentiation of a Stieltjes--Bochner integral]\label{t2.9}
Let $V$ be a Banach space and let $f:[a,b]\to V$ be Bochner
$g$-integrable.  Define
		\begin{displaymath}
			F(t)=\int_{[a,t)} f(s) \, d\mu_g(s),\qquad t\in[a,b].
	\end{displaymath}
	Then $F\in AC_g([a,b];V)$ and there exists a $g$-measurable set
	$N\subset [a,b)$ such that
	$\mu_g(N)=0$ and
	\begin{displaymath}
		F_g'(t)=f(t),\qquad t \in [a,b) \setminus N.
		\end{displaymath} 
\end{thm}
\begin{proof}
For completeness and to keep the paper self-contained, we give the details of
the differentiation argument from \cite[Theorem~2.9]{Fran2020}.  By
Lemma~\ref{lem:dominated-simple-approximation}, choose simple $V$-valued
functions $f_n$ such that $f_n(t)\to f(t)$ for $g$-almost every $t$ and
$\|f_n(t)\|_V\leq2\|f(t)\|_V$.  Every $f_n$ and
$h_n:=\|f-f_n\|_V$ is $g$-integrable.  Put
\begin{displaymath}
 F_n(t)=\int_{[a,t)}f_n(s)\,d\mu_g(s),
 \qquad
 H_n(t)=\int_{[a,t)}h_n(s)\,d\mu_g(s).
\end{displaymath}
Since $f_n$ has finite-dimensional range, this step can be reduced explicitly
to the scalar differentiation theorem.  Let
\begin{displaymath}
 Y_n:=\operatorname{span}(f_n([a,b)))
\end{displaymath}
and choose a basis $e_1,\ldots,e_{m_n}$ of $Y_n$.  Writing
\begin{displaymath}
 f_n(t)=\sum_{j=1}^{m_n} f_{n,j}(t)e_j,
\qquad
 F_n(t)=\sum_{j=1}^{m_n}
 \left(\int_{[a,t)} f_{n,j}(s)\,d\mu_g(s)\right)e_j,
\end{displaymath}
Theorem~2.4 of \cite{POUSO2015} applies to every scalar component
$f_{n,j}$.  Because there are only finitely many components, the union of
the corresponding exceptional sets is still $g$-null.  Outside that union,
every coordinate of the difference quotient converges to the corresponding
coordinate of $f_n(t)$; equivalence of norms on the finite-dimensional space
$Y_n$ therefore yields convergence in the norm of $Y_n$, and hence in the
norm of $V$.  Consequently there is a $g$-null set $N_n^F$ such that
\begin{displaymath}
 (F_n)'_g(t)=f_n(t),
 \qquad t\in[a,b)\setminus N_n^F.
\end{displaymath}
This is the componentwise content of the finite-rank differentiation theorem
\cite[Theorem~2.8]{Fran2020}.
Independently, the scalar differentiation theorem
\cite[Theorem~2.4]{POUSO2015}, applied to $h_n$, gives a $g$-null set
$N_n^H$ such that
\begin{displaymath}
 (H_n)'_g(t)=h_n(t),
 \qquad t\in[a,b)\setminus N_n^H.
\end{displaymath}
We set $N_n=N_n^F\cup N_n^H$.

Let $N$ be the union of the sets $N_n$, the exceptional set where
$f_n\not\to f$, and $C_g\cup N_g$.  This is a $g$-null set.  Fix
$t\in[a,b)\setminus(N\cup D_g)$.  If $s>t$ and $g(s)\ne g(t)$, then
\begin{displaymath}
\begin{aligned}
&\left\|
\frac{(F-F_n)(s)-(F-F_n)(t)}{g(s)-g(t)}
\right\|_V
\\
&\quad=
\left\|
\frac{1}{g(s)-g(t)}
\int_{[t,s)}(f(r)-f_n(r))\,d\mu_g(r)
\right\|_V
\\
&\quad\leq
\frac{1}{g(s)-g(t)}
\int_{[t,s)}h_n(r)\,d\mu_g(r)
=
\frac{H_n(s)-H_n(t)}{g(s)-g(t)}.
\end{aligned}
\end{displaymath}
For $s<t$, the same argument applies after reversing both increments.
Since $(H_n)'_g(t)=h_n(t)$, it follows that
\begin{displaymath}
\limsup_{s\to t}
\left\|
\frac{(F-F_n)(s)-(F-F_n)(t)}{g(s)-g(t)}
\right\|_V
\leq h_n(t),
\end{displaymath}
where the limit is taken through points for which $g(s)\ne g(t)$.

For such $s$, the triangle inequality gives
\begin{displaymath}
\begin{aligned}
\left\|
\frac{F(s)-F(t)}{g(s)-g(t)}-f(t)
\right\|_V
&\leq
\left\|
\frac{(F-F_n)(s)-(F-F_n)(t)}{g(s)-g(t)}
\right\|_V
\\
&\quad+
\left\|
\frac{F_n(s)-F_n(t)}{g(s)-g(t)}-f_n(t)
\right\|_V
+
\|f_n(t)-f(t)\|_V.
\end{aligned}
\end{displaymath}
Taking the upper limit as $s\to t$, using
$(F_n)'_g(t)=f_n(t)$ and
$\|f_n(t)-f(t)\|_V=h_n(t)$, yields
\begin{displaymath}
 \limsup_{s\to t}
 \left\|\frac{F(s)-F(t)}{g(s)-g(t)}-f(t)\right\|_V
 \leq 2h_n(t).
\end{displaymath}
Finally, $h_n(t)\to0$ because $t$ lies outside the exceptional set where
$f_n\not\to f$.  Hence $F'_g(t)=f(t)$.

If $t\in D_g$, the calculation is direct:
\begin{displaymath}
 F(t^+)-F(t)=\int_{\{t\}}f(s)\,d\mu_g(s)
 =f(t)\Delta^+g(t).
\end{displaymath}
The value of a Bochner class at an atom is well defined, and the atomic
formula therefore holds outside no additional exceptional set.  We have
proved $F'_g=f$ $g$-almost everywhere.

It remains to prove the stated absolute continuity.  Given
$\varepsilon>0$, choose $R>0$ so that
\begin{displaymath}
 \int_{\{\|f\|_V>R\}}\|f(s)\|_V\,d\mu_g(s)<\frac{\varepsilon}{2}.
\end{displaymath}
Set $\delta=\varepsilon/(2R)$; if $f=0$ almost everywhere, any positive
$\delta$ may be used.  Whenever $E$ is measurable and
$\mu_g(E)<\delta$,
\begin{displaymath}
 \int_E\|f(s)\|_V\,d\mu_g(s)
 \leq R\mu_g(E)
 +\int_{\{\|f\|_V>R\}}\|f(s)\|_V\,d\mu_g(s)
 <\varepsilon.
\end{displaymath}
Now let $(c_i,d_i)$, $i=1,\ldots,n$, be pairwise disjoint intervals with
$\sum_i(g(d_i)-g(c_i))<\delta$.  Their half-open versions are disjoint and
their union $E$ satisfies $\mu_g(E)<\delta$.  Therefore
\begin{displaymath}
 \sum_{i=1}^n\|F(d_i)-F(c_i)\|_V
 \leq\sum_{i=1}^n\int_{[c_i,d_i)}\|f(s)\|_V\,d\mu_g(s)
 =\int_E\|f(s)\|_V\,d\mu_g(s)<\varepsilon.
\end{displaymath}
This is the defining condition for $F\in AC_g([a,b];V)$.
\end{proof}

\begin{rem}
	Theorem~\ref{t2.9} identifies $F'_g$ on $[a,b)$ outside a
	$\mu_g$-null set.  No derivative at $b$ is required.  Notice also that the
	Radon--Nikod\'ym property does not enter this result: the density $f$ is
	already strongly measurable and Bochner integrable.
\end{rem}

We use throughout the standard estimate
\begin{displaymath}
 \left\|\int_A f\,d\mu\right\|\leq\int_A\|f\|\,d\mu
\end{displaymath}
for Bochner-integrable functions.
\begin{lem}[A consequence of Ehrling's lemma]\label{l7.6}
Let $V_1,V_2,V_3$ be Banach spaces such that
$V_1\Subset V_2$ and $V_2\hookrightarrow V_3$ continuously. If $p\geq1$,
then, for every $\varepsilon>0$, there exists $C_\varepsilon>0$ such that
\begin{displaymath}
	\|v\|_{V_2}^p
	\leq
	\varepsilon\|v\|_{V_1}^p
	+
	C_\varepsilon\|v\|_{V_3}^p,
	\qquad v\in V_1.
\end{displaymath}
\end{lem}
\begin{proof}
The usual form of Ehrling's lemma states that, for every $\eta>0$, there is
$C_\eta>0$ such that
\begin{displaymath}
 \|v\|_{V_2}\leq\eta\|v\|_{V_1}+C_\eta\|v\|_{V_3},
 \qquad v\in V_1;
\end{displaymath}
see \cite[Lemma~7.6]{Roubicek2013}.  Raising this inequality to the power
$p$ and using $(r+s)^p\leq2^{p-1}(r^p+s^p)$ gives
\begin{displaymath}
 \|v\|_{V_2}^p
 \leq2^{p-1}\eta^p\|v\|_{V_1}^p
   +2^{p-1}C_\eta^p\|v\|_{V_3}^p.
\end{displaymath}
Given $\varepsilon>0$, choose
$\eta=(\varepsilon/2^{p-1})^{1/p}$ and put
$C_\varepsilon=2^{p-1}C_\eta^p$.
\end{proof}
\section{A fundamental theorem for Stieltjes--Bochner integrals}
\label{sec:fundamental-theorem}

We now extend Theorem~\ref{t5.4} to Banach-valued functions.  The appropriate
assumption on the range space is the Radon--Nikod\'ym property.  In particular,
the result applies when the range space is reflexive.

Recall that a Banach space $Y$ has the \emph{Radon--Nikod\'ym property} if, for
every finite measure space $(X,\Sigma,\mu)$, every countably additive
$Y$-valued measure $\nu$ of bounded variation satisfying $\nu\ll\mu$ admits
a Bochner-integrable density $v\in L^1(X,\mu;Y)$, namely
$\nu(A)=\int_Av\,d\mu$ for all $A\in\Sigma$. This vector-measure formulation
is given in \cite[Ch.~III, Sect.~1]{DiestelUhl1977}; the fact that every
reflexive Banach space has this property is proved in
\cite[Ch.~III, Sect.~2]{DiestelUhl1977}.

The proof below uses the RNP only once: after showing that the vector measure
$\nu_u$ is absolutely continuous with respect to $\mu_g$.  All preceding
steps, including the construction of $\nu_u$ and the identification of its
total variation, are valid in an arbitrary Banach space.  Keeping these steps separate makes both the scope and the limitation of
the result explicit.

We start with two measure-theoretic lemmas.  Including their proofs makes the
subsequent use of vector measures transparent.

\begin{lem}[Approximation by the interval algebra]
	\label{lem:interval-algebra-density}
	Let $\lambda$ be a finite Borel measure on $[a,b)$, and let $\mathcal A$
	be the algebra of all finite disjoint unions of half-open intervals
	$[c,d)$ with $a\leq c<d\leq b$, together with the empty set. For
	$E,A\subset [a,b)$, their symmetric difference is defined by
	\begin{displaymath}
		E\mathbin{\triangle}A
		:=
		(E\setminus A)\cup(A\setminus E).
	\end{displaymath}
	Then, for every Borel set $E\subset[a,b)$ and every $\varepsilon>0$,
	there exists $A\in\mathcal A$ such that
	\begin{displaymath}
		\lambda(E\mathbin{\triangle}A)<\varepsilon.
	\end{displaymath}
\end{lem}

\begin{proof}
	All complements in this proof are taken relative to $[a,b)$. Define
	$\mathcal D$ as the family of all sets $E\in\mathcal B([a,b))$ with the
	following property: for every $\varepsilon>0$, there exists
	$A\in\mathcal A$ such that
	\begin{displaymath}
		\lambda(E\mathbin{\triangle}A)<\varepsilon.
	\end{displaymath}
	We prove that $\mathcal D$ is a $\sigma$-algebra containing $\mathcal A$.
	
	\medskip
	\noindent
	\emph{Step 1: $\mathcal A\subset\mathcal D$.}
	Let $E\in\mathcal A$. Choosing $A=E$, we have
	\begin{displaymath}
		E\mathbin{\triangle}A
		=
		E\mathbin{\triangle}E
		=
		\varnothing,
	\end{displaymath}
	and therefore
	\begin{displaymath}
		\lambda(E\mathbin{\triangle}A)=0.
	\end{displaymath}
	Thus $E\in\mathcal D$.
	
	\medskip
	\noindent
	\emph{Step 2: $\mathcal D$ is closed under complements.}
	Let $E\in\mathcal D$ and let $\varepsilon>0$. Choose
	$A\in\mathcal A$ such that
	\begin{displaymath}
		\lambda(E\mathbin{\triangle}A)<\varepsilon.
	\end{displaymath}
	Since $\mathcal A$ is an algebra, $A^c\in\mathcal A$. Moreover,
	\begin{displaymath}
		E^c\mathbin{\triangle}A^c
		=
		E\mathbin{\triangle}A.
	\end{displaymath}
	Indeed, a point belongs to exactly one of $E^c$ and $A^c$ if and only if
	it belongs to exactly one of $E$ and $A$. Hence
	\begin{displaymath}
		\lambda(E^c\mathbin{\triangle}A^c)
		=
		\lambda(E\mathbin{\triangle}A)
		<
		\varepsilon.
	\end{displaymath}
	Thus $E^c\in\mathcal D$.
	
	\medskip
	\noindent
	\emph{Step 3: $\mathcal D$ is closed under finite unions.}
	Let $E_1,E_2\in\mathcal D$ and let $\varepsilon>0$. For $i=1,2$,
	choose $A_i\in\mathcal A$ such that
	\begin{displaymath}
		\lambda(E_i\mathbin{\triangle}A_i)
		<
		\frac{\varepsilon}{2}.
	\end{displaymath}
	Since $\mathcal A$ is an algebra, $A_1\cup A_2\in\mathcal A$. We have
	\begin{displaymath}
		(E_1\cup E_2)\mathbin{\triangle}(A_1\cup A_2)
		\subset
		(E_1\mathbin{\triangle}A_1)
		\cup
		(E_2\mathbin{\triangle}A_2).
	\end{displaymath}
	To prove the inclusion, take a point in the set on the left. If it belongs
	to $E_1\cup E_2$ but not to $A_1\cup A_2$, then it belongs to
	$E_i\setminus A_i$ for at least one $i\in\{1,2\}$. If it belongs to
	$A_1\cup A_2$ but not to $E_1\cup E_2$, then it belongs to
	$A_i\setminus E_i$ for at least one $i\in\{1,2\}$. In either case it
	belongs to the set on the right. Therefore,
	\begin{displaymath}
		\begin{aligned}
			&\lambda\bigl(
			(E_1\cup E_2)\mathbin{\triangle}(A_1\cup A_2)
			\bigr)
			\\
			&\qquad\leq
			\lambda(E_1\mathbin{\triangle}A_1)
			+
			\lambda(E_2\mathbin{\triangle}A_2)
			<
			\varepsilon.
		\end{aligned}
	\end{displaymath}
	Thus $E_1\cup E_2\in\mathcal D$. By induction, $\mathcal D$ is closed
	under arbitrary finite unions.
	
	\medskip
	\noindent
	\emph{Step 4: $\mathcal D$ is closed under countable unions.}
	Let $(E_j)_{j\geq1}$ be a sequence in $\mathcal D$, and set
	\begin{displaymath}
		E:=\bigcup_{j=1}^{\infty}E_j,
		\qquad
		E^{(N)}:=\bigcup_{j=1}^{N}E_j.
	\end{displaymath}
	By Step 3, $E^{(N)}\in\mathcal D$ for every $N$. Moreover,
	$E^{(N)}\uparrow E$. Since $\lambda$ is finite, continuity from below gives
	\begin{displaymath}
		\lambda(E)
		=
		\lim_{N\to\infty}\lambda(E^{(N)}).
	\end{displaymath}
	Because $E^{(N)}\subset E$, it follows that
	\begin{displaymath}
		\lambda\bigl(E\setminus E^{(N)}\bigr)
		=
		\lambda(E)-\lambda(E^{(N)})
		\longrightarrow0.
	\end{displaymath}
	Fix $\varepsilon>0$. Choose $N$ such that
	\begin{displaymath}
		\lambda\bigl(E\setminus E^{(N)}\bigr)
		<
		\frac{\varepsilon}{2}.
	\end{displaymath}
	Since $E^{(N)}\in\mathcal D$, choose $A\in\mathcal A$ such that
	\begin{displaymath}
		\lambda\bigl(E^{(N)}\mathbin{\triangle}A\bigr)
		<
		\frac{\varepsilon}{2}.
	\end{displaymath}
	We claim that
	\begin{displaymath}
		E\mathbin{\triangle}A
		\subset
		\bigl(E\setminus E^{(N)}\bigr)
		\cup
		\bigl(E^{(N)}\mathbin{\triangle}A\bigr).
	\end{displaymath}
	Indeed, if $x\in E\setminus A$, then either $x\notin E^{(N)}$, in which
	case $x\in E\setminus E^{(N)}$, or $x\in E^{(N)}\setminus A$. If
	$x\in A\setminus E$, then, since $E^{(N)}\subset E$, we also have
	$x\in A\setminus E^{(N)}$. Thus the claimed inclusion holds, and hence
	\begin{displaymath}
		\begin{aligned}
			\lambda(E\mathbin{\triangle}A)
			&\leq
			\lambda\bigl(E\setminus E^{(N)}\bigr)
			+
			\lambda\bigl(E^{(N)}\mathbin{\triangle}A\bigr)
			\\
			&<
			\frac{\varepsilon}{2}
			+
			\frac{\varepsilon}{2}
			=
			\varepsilon.
		\end{aligned}
	\end{displaymath}
	Therefore $E\in\mathcal D$, and $\mathcal D$ is closed under countable
	unions.
	
	\medskip
	\noindent
	\emph{Step 5: $\mathcal A$ generates the Borel $\sigma$-algebra of
		$[a,b)$.}
	Every set in $\mathcal A$ is Borel, so
	\begin{displaymath}
		\sigma(\mathcal A)
		\subset
		\mathcal B([a,b)).
	\end{displaymath}
	For the converse inclusion, let $\mathcal G$ be the family consisting of
	all intervals of the following three types:
	\begin{displaymath}
		[a,q),
		\qquad
		(p,q),
		\qquad
		(p,b),
	\end{displaymath}
	where $p,q\in\mathbb Q$ and $a<p<q<b$ whenever the corresponding
	endpoints occur. The family $\mathcal G$ is countable and is a basis for
	the relative topology of $[a,b)$. We now verify that every member of
	$\mathcal G$ belongs to $\sigma(\mathcal A)$. Intervals of the form
	$[a,q)$ already belong to $\mathcal A$. If $p<q$, choose a sequence
	$(p_n)$ in $(p,q)$ such that $p_n\downarrow p$. Then
	\begin{displaymath}
		(p,q)
		=
		\bigcup_{n=1}^{\infty}[p_n,q),
	\end{displaymath}
	so $(p,q)\in\sigma(\mathcal A)$. Similarly,
	\begin{displaymath}
		(p,b)
		=
		\bigcup_{n=1}^{\infty}[p_n,b),
	\end{displaymath}
	and hence $(p,b)\in\sigma(\mathcal A)$. Therefore every member of
	$\mathcal G$ belongs to $\sigma(\mathcal A)$. Since $\mathcal G$ is a
	countable basis, every relatively open subset of $[a,b)$ is a countable
	union of members of $\mathcal G$ and consequently belongs to
	$\sigma(\mathcal A)$. Thus
	\begin{displaymath}
		\mathcal B([a,b))
		\subset
		\sigma(\mathcal A).
	\end{displaymath}
	Consequently,
	\begin{displaymath}
		\sigma(\mathcal A)
		=
		\mathcal B([a,b)).
	\end{displaymath}
	
	Since $\mathcal D$ is a $\sigma$-algebra containing $\mathcal A$, we have
	\begin{displaymath}
		\mathcal B([a,b))
		=
		\sigma(\mathcal A)
		\subset
		\mathcal D.
	\end{displaymath}
	The reverse inclusion follows from the definition of $\mathcal D$. Hence
	\begin{displaymath}
		\mathcal D
		=
		\mathcal B([a,b)),
	\end{displaymath}
	which proves the result.
\end{proof}

\begin{lem}[Variation of a Bochner density]
	\label{lem:variation-Bochner-density}
	Let $(X,\Sigma,\mu)$ be a finite measure space, let $Y$ be a Banach space,
	and let $v\in L^1(X,\mu;Y)$. Define
	\begin{displaymath}
		\nu_v(E)
		:=
		\int_E v\,\mathrm{d}\mu,
		\qquad E\in\Sigma.
	\end{displaymath}
	Then $\nu_v:\Sigma\to Y$ is a countably additive vector measure. For
	$E\in\Sigma$, let $\mathcal P(E)$ denote the family of all finite
	measurable partitions of $E$. Thus an element of $\mathcal P(E)$ is a
	finite family $\{E_1,\ldots,E_n\}$ of pairwise disjoint sets in
	$\Sigma$ whose union is $E$. The total variation of $\nu_v$ is defined by
	\begin{displaymath}
		|\nu_v|(E)
		:=
		\sup_{\{E_1,\ldots,E_n\}\in\mathcal P(E)}
		\sum_{i=1}^{n}\|\nu_v(E_i)\|_Y.
	\end{displaymath}
	Then
	\begin{equation}
		\label{eq:variation-Bochner-density}
		|\nu_v|(E)
		=
		\int_E\|v(s)\|_Y\,\mathrm{d}\mu(s),
		\qquad E\in\Sigma.
	\end{equation}
\end{lem}

\begin{proof}
	We prove each assertion separately.
	
	\medskip
	\noindent
	\emph{Step 1: $\nu_v$ is well defined.}
	For every $E\in\Sigma$, the function $\mathbf{1}_E v$ is Bochner
	integrable because
	\begin{displaymath}
		\int_X\|\mathbf{1}_E(s)v(s)\|_Y\,\mathrm{d}\mu(s)
		=
		\int_E\|v(s)\|_Y\,\mathrm{d}\mu(s)
		\leq
		\int_X\|v(s)\|_Y\,\mathrm{d}\mu(s)
		<
		\infty.
	\end{displaymath}
	Thus $\nu_v(E)$ is well defined.
	
	\medskip
	\noindent
	\emph{Step 2: $\nu_v$ is countably additive.}
	Let $(E_j)_{j\geq1}$ be pairwise disjoint sets in $\Sigma$, and set
	\begin{displaymath}
		E:=\bigcup_{j=1}^{\infty}E_j.
	\end{displaymath}
	For every $N\geq1$, finite additivity of the Bochner integral gives
	\begin{displaymath}
		\sum_{j=1}^{N}\nu_v(E_j)
		=
		\int_{\bigcup_{j=1}^{N}E_j}v\,\mathrm{d}\mu.
	\end{displaymath}
	Consequently,
	\begin{displaymath}
		\begin{aligned}
			\left\|
			\nu_v(E)-\sum_{j=1}^{N}\nu_v(E_j)
			\right\|_Y
			&=
			\left\|
			\int_{\bigcup_{j>N}E_j}v\,\mathrm{d}\mu
			\right\|_Y
			\\
			&\leq
			\int_{\bigcup_{j>N}E_j}\|v(s)\|_Y\,\mathrm{d}\mu(s).
		\end{aligned}
	\end{displaymath}
	The scalar measure
	\begin{displaymath}
		\rho(F)
		:=
		\int_F\|v(s)\|_Y\,\mathrm{d}\mu(s)
	\end{displaymath}
	is finite. The sets $\bigcup_{j>N}E_j$ decrease to the empty set, so
	continuity from above gives
	\begin{displaymath}
		\int_{\bigcup_{j>N}E_j}\|v(s)\|_Y\,\mathrm{d}\mu(s)
		\longrightarrow0.
	\end{displaymath}
	Hence
	\begin{displaymath}
		\nu_v(E)
		=
		\sum_{j=1}^{\infty}\nu_v(E_j)
	\end{displaymath}
	with convergence in $Y$.
	
	\medskip
	\noindent
	\emph{Step 3: upper estimate for the total variation.}
	Let
	\begin{displaymath}
		E=\bigcup_{i=1}^{n}E_i
	\end{displaymath}
	be a finite measurable partition of $E$. Then
	\begin{displaymath}
		\begin{aligned}
			\sum_{i=1}^{n}\|\nu_v(E_i)\|_Y
			&=
			\sum_{i=1}^{n}
			\left\|
			\int_{E_i}v\,\mathrm{d}\mu
			\right\|_Y
			\\
			&\leq
			\sum_{i=1}^{n}
			\int_{E_i}\|v(s)\|_Y\,\mathrm{d}\mu(s)
			\\
			&=
			\int_E\|v(s)\|_Y\,\mathrm{d}\mu(s).
		\end{aligned}
	\end{displaymath}
	Taking the supremum over all finite measurable partitions of $E$ gives
	\begin{displaymath}
		|\nu_v|(E)
		\leq
		\int_E\|v(s)\|_Y\,\mathrm{d}\mu(s).
	\end{displaymath}
	In particular, $\nu_v$ has bounded variation.
	
	\medskip
	\noindent
	\emph{Step 4: equality for simple functions.}
	Assume that
	\begin{displaymath}
		v
		=
		\sum_{j=1}^{m}y_j\mathbf{1}_{F_j},
	\end{displaymath}
	where $y_j\in Y$ and the measurable sets $F_j$ are pairwise disjoint.
	For a fixed $E\in\Sigma$, the sets
	\begin{displaymath}
		E\cap F_1,\ldots,E\cap F_m,
		\qquad
		E\setminus\bigcup_{j=1}^{m}F_j
	\end{displaymath}
	form a finite measurable partition of $E$, after empty sets are discarded.
	For every $j$,
	\begin{displaymath}
		\nu_v(E\cap F_j)
		=
		y_j\mu(E\cap F_j),
	\end{displaymath}
	and therefore
	\begin{displaymath}
		\|\nu_v(E\cap F_j)\|_Y
		=
		\|y_j\|_Y\mu(E\cap F_j).
	\end{displaymath}
	The vector measure vanishes on
	$E\setminus\bigcup_{j=1}^{m}F_j$. Hence, by the definition of total
	variation,
	\begin{displaymath}
		\begin{aligned}
			|\nu_v|(E)
			&\geq
			\sum_{j=1}^{m}\|\nu_v(E\cap F_j)\|_Y
			\\
			&=
			\sum_{j=1}^{m}\|y_j\|_Y\mu(E\cap F_j)
			\\
			&=
			\int_E\|v(s)\|_Y\,\mathrm{d}\mu(s).
		\end{aligned}
	\end{displaymath}
	Combining this inequality with Step 3 proves
	\begin{displaymath}
		|\nu_v|(E)
		=
		\int_E\|v(s)\|_Y\,\mathrm{d}\mu(s)
	\end{displaymath}
	for every simple function $v$.
	
	\medskip
	\noindent
	\emph{Step 5: stability of total variation.}
	Let $\nu$ and $\eta$ be $Y$-valued vector measures of bounded variation.
	For every finite measurable partition $E=\bigcup_{i=1}^{n}E_i$,
	\begin{displaymath}
		\begin{aligned}
			\sum_{i=1}^{n}\|\nu(E_i)\|_Y
			&\leq
			\sum_{i=1}^{n}\|\nu(E_i)-\eta(E_i)\|_Y
			+
			\sum_{i=1}^{n}\|\eta(E_i)\|_Y
			\\
			&\leq
			|\nu-\eta|(E)+|\eta|(E).
		\end{aligned}
	\end{displaymath}
	Taking the supremum over all such partitions gives
	\begin{displaymath}
		|\nu|(E)
		\leq
		|\nu-\eta|(E)+|\eta|(E).
	\end{displaymath}
	Thus
	\begin{displaymath}
		|\nu|(E)-|\eta|(E)
		\leq
		|\nu-\eta|(E).
	\end{displaymath}
	Interchanging $\nu$ and $\eta$ yields
	\begin{displaymath}
		|\eta|(E)-|\nu|(E)
		\leq
		|\nu-\eta|(E).
	\end{displaymath}
	Therefore,
	\begin{displaymath}
		\bigl||\nu|(E)-|\eta|(E)\bigr|
		\leq
		|\nu-\eta|(E).
	\end{displaymath}
	
	\medskip
	\noindent
	\emph{Step 6: passage to a general Bochner integrable density.}
	Choose simple functions $v_n$ such that
	\begin{displaymath}
		\|v_n-v\|_{L^1(X,\mu;Y)}
		\longrightarrow0.
	\end{displaymath}
	For every $E\in\Sigma$, Step 4 gives
	\begin{displaymath}
		|\nu_{v_n}|(E)
		=
		\int_E\|v_n(s)\|_Y\,\mathrm{d}\mu(s).
	\end{displaymath}
	Moreover,
	\begin{displaymath}
		\nu_v-\nu_{v_n}
		=
		\nu_{v-v_n}.
	\end{displaymath}
	By Step 3, applied to $v-v_n$,
	\begin{displaymath}
		|\nu_v-\nu_{v_n}|(E)
		\leq
		\int_E\|v(s)-v_n(s)\|_Y\,\mathrm{d}\mu(s).
	\end{displaymath}
	Using Step 5, we obtain
	\begin{displaymath}
		\begin{aligned}
			\bigl||\nu_v|(E)-|\nu_{v_n}|(E)\bigr|
			&\leq
			|\nu_v-\nu_{v_n}|(E)
			\\
			&\leq
			\int_E\|v(s)-v_n(s)\|_Y\,\mathrm{d}\mu(s)
			\\
			&\leq
			\|v-v_n\|_{L^1(X,\mu;Y)}.
		\end{aligned}
	\end{displaymath}
	Hence
	\begin{displaymath}
		|\nu_{v_n}|(E)
		\longrightarrow
		|\nu_v|(E).
	\end{displaymath}
	On the other hand, the reverse triangle inequality gives
	\begin{displaymath}
		\bigl|\|v_n(s)\|_Y-\|v(s)\|_Y\bigr|
		\leq
		\|v_n(s)-v(s)\|_Y
	\end{displaymath}
	for almost every $s\in X$. Therefore,
	\begin{displaymath}
		\begin{aligned}
			&\left|
			\int_E\|v_n(s)\|_Y\,\mathrm{d}\mu(s)
			-
			\int_E\|v(s)\|_Y\,\mathrm{d}\mu(s)
			\right|
			\\
			&\qquad\leq
			\int_E\|v_n(s)-v(s)\|_Y\,\mathrm{d}\mu(s)
			\longrightarrow0.
		\end{aligned}
	\end{displaymath}
	Passing to the limit in
	\begin{displaymath}
		|\nu_{v_n}|(E)
		=
		\int_E\|v_n(s)\|_Y\,\mathrm{d}\mu(s)
	\end{displaymath}
	gives
	\begin{displaymath}
		|\nu_v|(E)
		=
		\int_E\|v(s)\|_Y\,\mathrm{d}\mu(s).
	\end{displaymath}
	This proves \eqref{eq:variation-Bochner-density}.
\end{proof}

A left-continuous curve $u:[a,b]\to Y$ of bounded variation defines
an additive set function on the algebra of finite disjoint unions of
half-open intervals by setting
\begin{displaymath}
	\nu_u([c,d)):=u(d)-u(c).
\end{displaymath}
The bounded variation of $u$ provides a uniform bound for the variation
of this set function and allows it to be extended to a vector measure.
The use of intervals of the form $[c,d)$ is consistent with the
left-continuous convention adopted throughout the paper. In particular,
the mass assigned to a singleton $\{t\}$ is
\begin{displaymath}
	\nu_u(\{t\})=u(t^+)-u(t)=\Delta^+u(t),
\end{displaymath}
so atoms record right jumps. Applied to the derivator $g$, this convention
gives
\begin{displaymath}
	\mu_g(\{t\})=\Delta^+g(t).
\end{displaymath}

\begin{pro}[Vector measure associated with a left-continuous BV curve]
	\label{prop:BV-vector-measure}
	Let $Y$ be a Banach space, let $u:[a,b]\to Y$ be left-continuous and of
	bounded variation, and define
	\begin{displaymath}
		V(t)=\operatorname{Var}_{[a,t]}(u),\qquad t\in[a,b].
	\end{displaymath}
	There is a unique countably additive $Y$-valued Borel measure $\nu_u$ on
	$[a,b)$ such that
	\begin{equation}\label{eq:BV-vector-intervals}
		\nu_u([c,d))=u(d)-u(c),
		\qquad a\leq c<d\leq b.
	\end{equation}
	It has bounded variation and
	\begin{equation}\label{eq:BV-vector-variation}
		|\nu_u|=\mu_V,
	\end{equation}
	where $\mu_V$ is the positive Lebesgue--Stieltjes measure determined by
	$\mu_V([c,d))=V(d)-V(c)$.
\end{pro}

\begin{proof}
	The additivity of the variation on adjacent intervals gives
	\begin{displaymath}
		V(d)-V(c)=\operatorname{Var}_{[c,d]}(u).
	\end{displaymath}
	Since $u$ is left-continuous, so is $V$.  Indeed, if $t_n\uparrow t$, then
	$V(t_n)\leq V(t)$ and the increasing limit $L=\lim_nV(t_n)$ exists.  Given a
	partition $a=r_0<\cdots<r_m=t$, choose $n$ so large that $t_n>r_{m-1}$.
	The triangle inequality gives
	\begin{displaymath}
		\sum_{j=1}^{m}\|u(r_j)-u(r_{j-1})\|_Y
		\leq V(t_n)+\|u(t)-u(t_n)\|_Y.
	\end{displaymath}
	After letting $n\to\infty$ and then taking the supremum over the partitions,
	we obtain $V(t)\leq L$.  Hence $V(t_n)\to V(t)$, and the usual
	Lebesgue--Stieltjes construction with half-open intervals defines $\mu_V$.
	
	Let $\mathcal A$ be the algebra of finite disjoint unions of half-open
	subintervals of $[a,b)$.  For
	$A=\bigcup_{j=1}^m[c_j,d_j)$, define
	\begin{displaymath}
		\nu_0(A)=\sum_{j=1}^m\bigl(u(d_j)-u(c_j)\bigr).
	\end{displaymath}
	We first verify that this definition is independent of the chosen interval
	decomposition. Suppose that the same set $A$ is represented as
	\begin{displaymath}
		A=\bigcup_{j=1}^m[c_j,d_j)
		=\bigcup_{k=1}^n[\alpha_k,\beta_k),
	\end{displaymath}
	where the intervals in each family are pairwise disjoint. Collect all the
	endpoints $c_j,d_j,\alpha_k,\beta_k$ and arrange the distinct ones in
	increasing order. The consecutive half-open intervals determined by these
	points form a common refinement of both decompositions. Refining an interval
	does not change its contribution, because for
	\begin{displaymath}
		c=q_0<q_1<\cdots<q_N=d
	\end{displaymath}
	the sum telescopes:
	\begin{displaymath}
		\sum_{\ell=1}^N
		\bigl(u(q_\ell)-u(q_{\ell-1})\bigr)
		=u(d)-u(c).
	\end{displaymath}
	Thus both decompositions produce the same value, and $\nu_0$ is well
	defined.
	
	The same common-refinement argument proves finite additivity. Indeed, if
	$A,B\in\mathcal A$ are disjoint, refine interval decompositions of $A$ and
	$B$ by all their endpoints. The resulting intervals form a disjoint interval
	decomposition of $A\cup B$, and the defining sum splits into the sum over
	$A$ and the sum over $B$. Hence
	\begin{displaymath}
		\nu_0(A\cup B)=\nu_0(A)+\nu_0(B).
	\end{displaymath}
	Moreover,
	\begin{equation}\label{eq:nu0-dominated-by-V}
		\|\nu_0(A)\|_Y
		\leq\sum_{j=1}^m\operatorname{Var}_{[c_j,d_j]}(u)
		=\mu_V(A).
	\end{equation}
	If $A_n\downarrow\varnothing$ in $\mathcal A$, then
	$\mu_V(A_n)\downarrow0$, and \eqref{eq:nu0-dominated-by-V} gives
	$\nu_0(A_n)\to0$.  This continuity at the empty set proves that $\nu_0$ is a
	vector premeasure.  Indeed, if $A=\bigcup_{j\geq1}A_j$ is a disjoint union
	with $A,A_j\in\mathcal A$, then
	\begin{displaymath}
		R_N:=A\setminus\bigcup_{j=1}^N A_j\downarrow\varnothing
	\end{displaymath}
	and finite additivity gives
	$\nu_0(A)-\sum_{j=1}^N\nu_0(A_j)=\nu_0(R_N)\to0$.
	
	The extension can be described explicitly. By
	Lemma~\ref{lem:interval-algebra-density}, for every Borel set $E$ there is a
	sequence $(A_n)$ in $\mathcal A$ such that
	\begin{displaymath}
		\mu_V(E\mathbin\triangle A_n)\longrightarrow0.
	\end{displaymath}
	For $n,m\in\mathbb N$, finite additivity on the algebra gives
	\begin{displaymath}
		\nu_0(A_n)-\nu_0(A_m)
		=\nu_0(A_n\setminus A_m)-\nu_0(A_m\setminus A_n).
	\end{displaymath}
	Indeed, decompose both $A_n$ and $A_m$ into their common intersection and
	their respective differences; the contributions of $A_n\cap A_m$ cancel.
	Using \eqref{eq:nu0-dominated-by-V}, we obtain
	\begin{displaymath}
		\|\nu_0(A_n)-\nu_0(A_m)\|_Y
		\leq\mu_V(A_n\mathbin\triangle A_m).
	\end{displaymath}
	Furthermore,
	\begin{displaymath}
		A_n\mathbin\triangle A_m
		\subset
		(A_n\mathbin\triangle E)\cup(E\mathbin\triangle A_m),
	\end{displaymath}
	so the right-hand side tends to zero. Hence $(\nu_0(A_n))$ is Cauchy in
	the Banach space $Y$.
	
	Define
	\begin{displaymath}
		\nu_u(E):=\lim_{n\to\infty}\nu_0(A_n).
	\end{displaymath}
	This definition is independent of the approximating sequence. Indeed, if
	$(B_n)\subset\mathcal A$ also satisfies
	$\mu_V(E\mathbin\triangle B_n)\to0$, then
	\begin{displaymath}
		\|\nu_0(A_n)-\nu_0(B_n)\|_Y
		\leq
		\mu_V(A_n\mathbin\triangle E)
		+\mu_V(E\mathbin\triangle B_n)
		\longrightarrow0.
	\end{displaymath}
	If $E\in\mathcal A$, the constant sequence $A_n=E$ shows that
	$\nu_u(E)=\nu_0(E)$, so the extension preserves the interval formula.
	
	Moreover,
	\begin{displaymath}
		|\mu_V(A_n)-\mu_V(E)|
		\leq\mu_V(A_n\mathbin\triangle E)\longrightarrow0.
	\end{displaymath}
	Passing to the limit in \eqref{eq:nu0-dominated-by-V} gives
	\begin{equation}\label{eq:nu-dominated-by-V}
		\|\nu_u(E)\|_Y\leq\mu_V(E),
		\qquad E\in\mathcal B([a,b)).
	\end{equation}
	
	We next prove finite additivity. Let $E$ and $F$ be disjoint Borel sets and
	choose $A_n,B_n\in\mathcal A$ approximating $E$ and $F$, respectively.
	Set
	\begin{displaymath}
		C_n:=A_n\setminus B_n,
		\qquad
		D_n:=B_n\setminus A_n.
	\end{displaymath}
	Then $C_n,D_n\in\mathcal A$ and $C_n\cap D_n=\varnothing$. Since
	$E\cap F=\varnothing$, one checks directly that
	\begin{displaymath}
		E\mathbin\triangle C_n
		\subset
		(E\mathbin\triangle A_n)\cup(F\mathbin\triangle B_n)
	\end{displaymath}
	and, similarly,
	\begin{displaymath}
		F\mathbin\triangle D_n
		\subset
		(F\mathbin\triangle B_n)\cup(E\mathbin\triangle A_n).
	\end{displaymath}
	Thus $C_n$ approximates $E$ and $D_n$ approximates $F$. Their union
	approximates $E\cup F$. Finite additivity of $\nu_0$ gives
	\begin{displaymath}
		\nu_0(C_n\cup D_n)=\nu_0(C_n)+\nu_0(D_n),
	\end{displaymath}
	and passage to the limit yields
	\begin{displaymath}
		\nu_u(E\cup F)=\nu_u(E)+\nu_u(F).
	\end{displaymath}
	Hence $\nu_u$ is finitely additive.
	
	Let now $(E_j)$ be pairwise disjoint and set
	\begin{displaymath}
		E:=\bigcup_{j=1}^{\infty}E_j.
	\end{displaymath}
	For every $N$, finite additivity and \eqref{eq:nu-dominated-by-V} give
	\begin{displaymath}
		\left\|\nu_u(E)-\sum_{j=1}^{N}\nu_u(E_j)\right\|_Y
		\leq
		\mu_V\left(\bigcup_{j>N}E_j\right).
	\end{displaymath}
	The sets on the right decrease to the empty set, and $\mu_V$ is finite;
	therefore the right-hand side tends to zero. This proves countable
	additivity.
	
	To prove uniqueness, let $\widehat\nu$ be another countably additive
	$Y$-valued Borel measure satisfying
	\eqref{eq:BV-vector-intervals}. By finite additivity, $\nu_u$ and
	$\widehat\nu$ agree on the algebra $\mathcal A$ of finite disjoint
	unions of half-open intervals.
	
	Recall that a nonempty family $\mathcal P$ of subsets of a set $X$ is
	called a $\pi$-system if it is closed under finite intersections. A
	$\lambda$-system, also called a Dynkin system, is a family $\mathcal D$
	of subsets of $X$ such that $X\in\mathcal D$, such that
	$F\setminus E\in\mathcal D$ whenever $E,F\in\mathcal D$ and $E\subset F$,
	and such that $\mathcal D$ is closed under countable disjoint unions.
	Dynkin's $\pi$--$\lambda$ theorem states that every $\lambda$-system
	containing a $\pi$-system $\mathcal P$ also contains the generated
	$\sigma$-algebra $\sigma(\mathcal P)$; see
	\cite[Ch.~1]{Kallenberg2002}. Since $\mathcal A$ is an algebra, it is in
	particular a $\pi$-system, and
	\begin{displaymath}
		\sigma(\mathcal A)=\mathcal B([a,b)).
	\end{displaymath}
	
	Fix $y'\in Y'$ and define the finite scalar signed or complex measure
	\begin{displaymath}
		\rho_{y'}(E)
		:=
		\left\langle y',\nu_u(E)-\widehat\nu(E)\right\rangle,
		\qquad E\in\mathcal B([a,b)).
	\end{displaymath}
	Since $\nu_u$ and $\widehat\nu$ agree on $\mathcal A$, one has
	\begin{displaymath}
		\rho_{y'}(A)=0,
		\qquad A\in\mathcal A.
	\end{displaymath}
	Consider
	\begin{displaymath}
		\mathcal D_{y'}
		:=
		\left\{E\in\mathcal B([a,b)):\rho_{y'}(E)=0\right\}.
	\end{displaymath}
	We verify that $\mathcal D_{y'}$ is a $\lambda$-system. First,
	$[a,b)\in\mathcal A$, and hence
	\begin{displaymath}
		\rho_{y'}([a,b))=0.
	\end{displaymath}
	Thus $[a,b)\in\mathcal D_{y'}$. Second, if
	$E,F\in\mathcal D_{y'}$ and $E\subset F$, then
	\begin{displaymath}
		\rho_{y'}(F\setminus E)
		=
		\rho_{y'}(F)-\rho_{y'}(E)
		=0,
	\end{displaymath}
	so $F\setminus E\in\mathcal D_{y'}$. Third, if $(E_n)_{n\geq1}$ is a
	sequence of pairwise disjoint sets in $\mathcal D_{y'}$, countable
	additivity gives
	\begin{displaymath}
		\rho_{y'}\left(\bigcup_{n=1}^{\infty}E_n\right)
		=
		\sum_{n=1}^{\infty}\rho_{y'}(E_n)
		=0.
	\end{displaymath}
	Therefore
	\begin{displaymath}
		\bigcup_{n=1}^{\infty}E_n\in\mathcal D_{y'}.
	\end{displaymath}
	This proves that $\mathcal D_{y'}$ is a $\lambda$-system.
	
	Because $\mathcal A\subset\mathcal D_{y'}$, Dynkin's
	$\pi$--$\lambda$ theorem yields
	\begin{displaymath}
		\mathcal B([a,b))
		=
		\sigma(\mathcal A)
		\subset
		\mathcal D_{y'}.
	\end{displaymath}
	Hence, for every $E\in\mathcal B([a,b))$,
	\begin{displaymath}
		\left\langle y',\nu_u(E)-\widehat\nu(E)\right\rangle=0.
	\end{displaymath}
	Since this holds for every $y'\in Y'$ and $Y'$ separates the points of
	$Y$, it follows that
	\begin{displaymath}
		\nu_u(E)=\widehat\nu(E),
		\qquad E\in\mathcal B([a,b)).
	\end{displaymath}
	Thus $\nu_u=\widehat\nu$. This is the standard vector-measure extension
	construction; compare \cite[Ch.~I, Sects.~1--2]{DiestelUhl1977}.

	We finally identify the total variation. Recall that, for a Borel set $E$,
	\begin{displaymath}
		|\nu_u|(E)
		:=
		\sup
		\sum_{j=1}^{m}\|\nu_u(E_j)\|_Y,
	\end{displaymath}
	where the supremum is taken over all finite Borel partitions
	$E=\bigcup_{j=1}^{m}E_j$ into pairwise disjoint sets. For every such
	partition, \eqref{eq:nu-dominated-by-V} gives
	\begin{displaymath}
		\sum_{j=1}^{m}\|\nu_u(E_j)\|_Y
		\leq
		\sum_{j=1}^{m}\mu_V(E_j)
		=\mu_V(E).
	\end{displaymath}
	Taking the supremum yields
	\begin{displaymath}
		|\nu_u|(E)\leq\mu_V(E).
	\end{displaymath}
	In particular, $|\nu_u|([a,b))<\infty$, so $\nu_u$ has bounded
	variation. We now verify directly that $|\nu_u|$ is countably additive.
	Let $(G_k)_{k\geq1}$ be pairwise disjoint Borel sets and set
	\begin{displaymath}
		G:=\bigcup_{k=1}^{\infty}G_k.
	\end{displaymath}
	Fix $N\in\mathbb N$ and $\eta>0$. For each
	$k\in\{1,\ldots,N\}$, choose a finite Borel partition of $G_k$ whose
	variation sum is larger than
	\begin{displaymath}
		|\nu_u|(G_k)-\frac{\eta}{N}.
	\end{displaymath}
	Combining these finitely many partitions and adding the remainder
	\begin{displaymath}
		G\setminus\bigcup_{k=1}^{N}G_k
	\end{displaymath}
	produces a finite Borel partition of $G$. Hence
	\begin{displaymath}
		|\nu_u|(G)
		\geq
		\sum_{k=1}^{N}|\nu_u|(G_k)-\eta.
	\end{displaymath}
	Letting $\eta\downarrow0$ and then $N\to\infty$ gives
	\begin{displaymath}
		|\nu_u|(G)
		\geq
		\sum_{k=1}^{\infty}|\nu_u|(G_k).
	\end{displaymath}
	
	For the reverse inequality, let
	\begin{displaymath}
		G=\bigcup_{i=1}^{m}F_i
	\end{displaymath}
	be an arbitrary finite Borel partition. Countable additivity of $\nu_u$
	gives, for each $i$,
	\begin{displaymath}
		\nu_u(F_i)
		=
		\sum_{k=1}^{\infty}\nu_u(F_i\cap G_k)
	\end{displaymath}
	with convergence in $Y$. Applying the triangle inequality to the partial
	sums and then passing to the limit yields
	\begin{displaymath}
		\|\nu_u(F_i)\|_Y
		\leq
		\sum_{k=1}^{\infty}\|\nu_u(F_i\cap G_k)\|_Y.
	\end{displaymath}
	Summing over $i$ and interchanging the finite sum with the nonnegative
	series, we obtain
	\begin{displaymath}
		\begin{aligned}
			\sum_{i=1}^{m}\|\nu_u(F_i)\|_Y
			&\leq
			\sum_{k=1}^{\infty}
			\sum_{i=1}^{m}\|\nu_u(F_i\cap G_k)\|_Y
			\\
			&\leq
			\sum_{k=1}^{\infty}|\nu_u|(G_k).
		\end{aligned}
	\end{displaymath}
	Taking the supremum over all finite Borel partitions of $G$ gives
	\begin{displaymath}
		|\nu_u|(G)
		\leq
		\sum_{k=1}^{\infty}|\nu_u|(G_k).
	\end{displaymath}
	Thus $|\nu_u|$ is a finite positive Borel measure.
	
	Fix $a\leq c<d\leq b$ and let
	\begin{displaymath}
		c=t_0<t_1<\cdots<t_m=d
	\end{displaymath}
	be an arbitrary point partition. The intervals $[t_{j-1},t_j)$ form a
	finite Borel partition of $[c,d)$. Therefore,
	\begin{displaymath}
		\begin{aligned}
			|\nu_u|([c,d))
			&\geq
			\sum_{j=1}^{m}\|\nu_u([t_{j-1},t_j))\|_Y
			\\
			&=
			\sum_{j=1}^{m}\|u(t_j)-u(t_{j-1})\|_Y.
		\end{aligned}
	\end{displaymath}
	Taking the supremum over all point partitions gives
	\begin{displaymath}
		|\nu_u|([c,d))
		\geq
		\operatorname{Var}_{[c,d]}(u)
		=\mu_V([c,d)).
	\end{displaymath}
	Together with the previously proved upper bound, this yields
	\begin{displaymath}
		|\nu_u|([c,d))=\mu_V([c,d))
	\end{displaymath}
	for every half-open interval. The two finite positive Borel measures
	$|\nu_u|$ and $\mu_V$ therefore agree on the algebra $\mathcal A$.
	Since $\mathcal A$ generates $\mathcal B([a,b))$, the uniqueness theorem
	for finite measures gives
	\begin{displaymath}
		|\nu_u|(E)=\mu_V(E)
	\end{displaymath}
	for every Borel set $E$. This proves
	\eqref{eq:BV-vector-variation}.
\end{proof}

\begin{thm}[Fundamental theorem of Stieltjes--Bochner calculus]\label{ftc}
Let $g:\mathbb R\to\mathbb R$ be a derivator and let
$(Y,\|\cdot\|)$ be a Banach space with the Radon--Nikod\'ym property. A
function $u:[a,b]\to Y$ belongs to $AC_g([a,b];Y)$ if and only if all of the
following conditions are satisfied:
\begin{enumerate}
\item the strong $g$-derivative $u'_g(t)\in Y$ exists for $g$-almost every
$t\in[a,b)$;
\item $u'_g\in L_g^1([a,b);Y)$;
\item
\begin{displaymath}
 u(t)=u(a)+\int_{[a,t)}u'_g(s)\,d\mu_g(s),
 \qquad t\in[a,b].
\end{displaymath}
\end{enumerate}
In this case the variation function
\begin{displaymath}
 V_u(t):=\operatorname{Var}_{[a,t]}(u),\qquad t\in[a,b],
\end{displaymath}
belongs to $AC_g([a,b])$ and satisfies
\begin{equation}\label{eq:metric-g-derivative}
 (V_u)'_g(t)=\|u'_g(t)\|_Y
 \quad\text{for $g$-almost every }t\in[a,b).
\end{equation}
\end{thm}

\begin{rem}[Why scalarization alone is insufficient]
	\label{rem:scalarization-insufficient}
	A natural first attempt is to fix $y'\in Y'$ and apply the scalar
	differentiation theorem to the function
	\begin{displaymath}
		t\longmapsto \langle y',u(t)\rangle.
	\end{displaymath}
	This procedure may identify the limits of the scalar difference
	quotients
	\begin{displaymath}
		\left\langle
		y',
		\frac{u(s)-u(t)}{g(s)-g(t)}
		\right\rangle.
	\end{displaymath}
	It does not, however, prove the existence of the strong
	$g$-derivative of $u$.
	
	There are several reasons for this. First, convergence of all scalar
	evaluations gives, at most, weak convergence of the vector difference
	quotients. The derivative in
	Definition~\ref{def:strong-g-derivative}, by contrast, is defined through
	convergence in the norm of $Y$. These two notions are not equivalent in
	an infinite-dimensional space. For example, the canonical basis
	$(e_n)$ of $\ell^2$ satisfies
	\begin{displaymath}
		e_n\rightharpoonup 0
	\end{displaymath}
	but
	\begin{displaymath}
		\|e_n\|_{\ell^2}=1
	\end{displaymath}
	for every $n$.
	
	Second, the exceptional null set supplied by the scalar differentiation
	theorem may depend on the functional $y'$. If the dual space is not
	separable, one cannot intersect the corresponding sets of full measure
	over all $y'\in Y'$ and still conclude that the resulting set has full
	measure. Thus the scalar identities need not hold simultaneously for
	every functional on one common set of full $\mu_g$-measure.
	
	Third, even if the scalar limits define, for each fixed $t$, a candidate
	weak limit, one must still prove that this candidate belongs to $Y$
	rather than merely to $Y''$, that it depends strongly measurably on
	$t$, and that the original difference quotients converge to it in norm.
	Extracting weakly convergent subsequences pointwise does not solve this
	problem, since the chosen subsequence may depend on $t$ and need not
	produce a single strongly measurable function.
	
	For these reasons, we first associate with $u$ the vector measure
	$\nu_u$. Under the absolute continuity condition
	\begin{displaymath}
		\nu_u\ll\mu_g,
	\end{displaymath}
	the Radon--Nikod\'ym property of $Y$ yields a Bochner integrable density
	$v$ satisfying
	\begin{displaymath}
		\nu_u(E)=\int_E v(s)\,d\mu_g(s)
	\end{displaymath}
	for every Borel set $E\subset[a,b)$. The vector-valued differentiation
	theorem then gives
	\begin{displaymath}
		u'_g(t)=v(t)
	\end{displaymath}
	in the norm of $Y$ for $\mu_g$-almost every $t$. This argument provides
	simultaneously a single vector-valued derivative, its strong
	measurability, and the norm convergence required in
	Definition~\ref{def:strong-g-derivative}.
\end{rem}

\begin{proof}
Assume first that $u\in AC_g([a,b];Y)$.  Lemma~\ref{lem:ACg-Cg} and
Proposition~\ref{prop:ACg-BPV} show that $u$ is $g$-continuous and of bounded
variation.  It is also left-continuous.  To see this, let $t_n\uparrow t$.
The left-continuity of $g$ gives $g(t_n)\to g(t)$, and $g$-continuity then
gives $u(t_n)\to u(t)$ in $Y$.

Set $V=V_u$, that is,
\begin{displaymath}
 V(t)=\operatorname{Var}_{[a,t]}(u),\qquad t\in[a,b].
\end{displaymath}
We claim that $V\in AC_g([a,b])$.  Let $\varepsilon>0$ and choose
$\delta>0$ from the $g$-absolute continuity of $u$, with
$\varepsilon/2$ in place of $\varepsilon$.  Consider pairwise disjoint
intervals $(c_i,d_i)$, $i=1,\ldots,n$, such that
\begin{displaymath}
 \sum_{i=1}^n\bigl(g(d_i)-g(c_i)\bigr)<\delta.
\end{displaymath}
For each $i$, choose an arbitrary partition
\begin{displaymath}
 c_i=t_{i,0}<t_{i,1}<\cdots<t_{i,m_i}=d_i.
\end{displaymath}
All the open intervals $(t_{i,j-1},t_{i,j})$ are pairwise disjoint and their
total $g$-increment is smaller than $\delta$.  Therefore
\begin{displaymath}
 \sum_{i=1}^n\sum_{j=1}^{m_i}
 \|u(t_{i,j})-u(t_{i,j-1})\|_Y<\frac{\varepsilon}{2}.
\end{displaymath}
Now fix $\eta>0$.  For every $i$ we may choose the partition so that its
variation sum is larger than
$\operatorname{Var}_{[c_i,d_i]}(u)-\eta/n$.  It follows that
\begin{displaymath}
 \sum_{i=1}^n\operatorname{Var}_{[c_i,d_i]}(u)
 \leq\frac{\varepsilon}{2}+\eta.
\end{displaymath}
Letting $\eta\downarrow0$ and using additivity of the variation on adjacent
intervals, we obtain
\begin{displaymath}
 \sum_{i=1}^n\bigl(V(d_i)-V(c_i)\bigr)
 =\sum_{i=1}^n\operatorname{Var}_{[c_i,d_i]}(u)
 \leq\frac{\varepsilon}{2}<\varepsilon.
\end{displaymath}
This proves the claim.

The scalar fundamental theorem, Theorem~\ref{t5.4}, now gives
\begin{equation}\label{eq:V-density}
 V(t)=\int_{[a,t)}V'_g(s)\,d\mu_g(s),
 \qquad t\in[a,b],
\end{equation}
because $V(a)=0$.  Hence, for every $a\leq c<d\leq b$,
\begin{displaymath}
	\begin{aligned}
		\mu_V([c,d))
		&=V(d)-V(c)\\
		&=\int_{[a,d)}V'_g(s)\,d\mu_g(s)
		-\int_{[a,c)}V'_g(s)\,d\mu_g(s)\\
		&=\int_{[c,d)}V'_g(s)\,d\mu_g(s).
	\end{aligned}
\end{displaymath}
The last equality follows from the disjoint decomposition
\begin{displaymath}
	[a,d)=[a,c)\cup[c,d)
\end{displaymath}
and the additivity of the Lebesgue integral.

We next verify that $V'_g$ is nonnegative at every point at which the
$g$-derivative exists. If $t\notin D_g$ and $s\to t$ through points with
$g(s)\neq g(t)$, then
\begin{displaymath}
	\frac{V(s)-V(t)}{g(s)-g(t)}\geq0.
\end{displaymath}
Indeed, when $s>t$, both the numerator and the denominator are
nonnegative, whereas when $s<t$, both are nonpositive. If $t\in D_g$,
then
\begin{displaymath}
	V'_g(t)
	=
	\frac{V(t^+)-V(t)}{g(t^+)-g(t)}
	\geq0,
\end{displaymath}
because both $V$ and $g$ are nondecreasing. Therefore, after assigning
the value zero on the exceptional $\mu_g$-null set where the derivative
is not defined, we may regard $V'_g$ as a nonnegative
$\mu_g$-measurable function.

Define, for every Borel set $E\subset[a,b)$,
\begin{displaymath}
	\lambda(E)
	:=
	\int_E V'_g(s)\,d\mu_g(s).
\end{displaymath}
Since $V'_g\in L_g^1([a,b))$ and $V'_g\geq0$
$\mu_g$-almost everywhere, $\lambda$ is a finite positive Borel
measure. Moreover, taking $t=b$ in \eqref{eq:V-density} and using
$V(a)=0$, we obtain
\begin{displaymath}
	\lambda([a,b))
	=
	\int_{[a,b)}V'_g(s)\,d\mu_g(s)
	=
	V(b)
	<
	\infty.
\end{displaymath}

The preceding interval identity shows that
\begin{displaymath}
	\lambda([c,d))=\mu_V([c,d))
\end{displaymath}
for every $a\leq c<d\leq b$. By finite additivity, the same equality
holds on the algebra $\mathcal A$ of finite disjoint unions of
half-open intervals. Since
\begin{displaymath}
	\sigma(\mathcal A)=\mathcal B([a,b)),
\end{displaymath}
and both $\lambda$ and $\mu_V$ are finite positive Borel measures,
the uniqueness theorem for finite measures on a generating algebra
(see \cite[Ch.~1]{Kallenberg2002}) yields
\begin{equation}\label{eq:variation-measure-absolute-continuity}
	\mu_V(E)
	=
	\lambda(E)
	=
	\int_E V'_g(s)\,d\mu_g(s)
\end{equation}
for every Borel set $E\subset[a,b)$.

Finally, let $E\subset[a,b)$ be Borel and suppose that
$\mu_g(E)=0$. Then the integral of the $\mu_g$-integrable function
$V'_g$ over $E$ vanishes, and \eqref{eq:variation-measure-absolute-continuity}
gives
\begin{displaymath}
	\mu_V(E)
	=
	\int_E V'_g(s)\,d\mu_g(s)
	=
	0.
\end{displaymath}
Hence
\begin{displaymath}
	\mu_V\ll\mu_g.
\end{displaymath}

Let $\nu_u$ be the vector measure supplied by
Proposition~\ref{prop:BV-vector-measure}. Since
$|\nu_u|=\mu_V$, equation
\eqref{eq:variation-measure-absolute-continuity} implies
$\nu_u\ll\mu_g$ on the Borel $\sigma$-algebra of $[a,b)$.

Apply the Radon--Nikod\'ym property of $Y$ to the finite Borel measure
space
\begin{displaymath}
 ([a,b),\mathcal B([a,b)),\mu_g).
\end{displaymath}
There exists a Bochner-integrable function $v:[a,b)\to Y$ such that
\begin{equation}\label{eq:vector-radon-nikodym-u}
 \nu_u(E)=\int_E v(s)\,d\mu_g(s)
\end{equation}
for every Borel set $E\subset[a,b)$. We may regard $v$ as an element of
$L_g^1([a,b);Y)$ and extend the identity to the $\mu_g$-completion as
follows.  If a completed-measurable set is written as $E\cup N$, where $E$
is Borel and $N$ is contained in a Borel $\mu_g$-null set, define
\begin{displaymath}
 \overline\nu_u(E\cup N):=\nu_u(E).
\end{displaymath}
This definition is independent of the chosen representation: the symmetric
difference of two possible Borel parts is $\mu_g$-null, and hence is also
$\nu_u$-null because $\nu_u\ll\mu_g$.  The completed Bochner integral is
extended in the same way, and its value is unchanged on adding a subset of a
$\mu_g$-null set.  Therefore
\begin{displaymath}
 \overline\nu_u(A)=\int_A v(s)\,d\mu_g(s)
\end{displaymath}
for every $g$-measurable set $A\subset[a,b)$.  We henceforth use the same
symbol $\nu_u$ for this unique completion.

Taking $E=[a,t)$ and using
\eqref{eq:BV-vector-intervals}, we find
\begin{displaymath}
 u(t)=u(a)+\int_{[a,t)}v(s)\,d\mu_g(s),
 \qquad t\in[a,b].
\end{displaymath}
Theorem~\ref{t2.9} applies to this indefinite Bochner integral and gives
\begin{displaymath}
 u'_g(t)=v(t)
 \quad\text{for $g$-almost every }t\in[a,b).
\end{displaymath}
Lemma~\ref{lem:variation-Bochner-density} and
$|\nu_u|=\mu_V$ further give, for every $g$-measurable
$E\subset[a,b)$,
\begin{displaymath}
 \int_E\|v(s)\|_Y\,d\mu_g(s)
 =|\nu_u|(E)
 =\mu_V(E)
 =\int_E V'_g(s)\,d\mu_g(s).
\end{displaymath}
Thus the two nonnegative functions $\|v(\cdot)\|_Y$ and $V'_g$ are
integrable and define the same finite scalar measure on the completed
$\sigma$-algebra:
\begin{displaymath}
 E\longmapsto
 \int_E\|v(s)\|_Y\,d\mu_g(s)
 =
 \int_E V'_g(s)\,d\mu_g(s).
\end{displaymath}
The uniqueness clause in the scalar Radon--Nikod\'ym theorem therefore gives
\begin{displaymath}
 \|v(t)\|_Y=V'_g(t)
 \quad\text{for $g$-almost every }t\in[a,b).
\end{displaymath}
Since $v=u'_g$ almost everywhere, this proves
\eqref{eq:metric-g-derivative}.
Thus the three conditions in the statement are satisfied.

Conversely, assume that those conditions hold and put $v=u'_g$.  We first
recall why the integral of $\|v\|_Y$ is absolutely continuous with respect to
$\mu_g$.  Given $\varepsilon>0$, choose $R>0$ such that
\begin{displaymath}
 \int_{\{\|v\|_Y>R\}}\|v(s)\|_Y\,d\mu_g(s)<\frac{\varepsilon}{2}.
\end{displaymath}
If $R>0$, take $\delta=\varepsilon/(2R)$; if $v=0$ almost everywhere, any
positive $\delta$ will do.  For every measurable $E$ with
$\mu_g(E)<\delta$,
\begin{displaymath}
 \int_E\|v(s)\|_Y\,d\mu_g(s)
 \leq R\mu_g(E)
 +\int_{\{\|v\|_Y>R\}}\|v(s)\|_Y\,d\mu_g(s)
 <\varepsilon.
\end{displaymath}

Let $(a_i,b_i)$, $i=1,\ldots,n$, be pairwise disjoint intervals with
\begin{displaymath}
 \sum_{i=1}^n\bigl(g(b_i)-g(a_i)\bigr)<\delta,
\end{displaymath}
and set $E=\bigcup_{i=1}^n[a_i,b_i)$.  The half-open intervals are disjoint,
so
\begin{displaymath}
 \mu_g(E)=\sum_{i=1}^n\bigl(g(b_i)-g(a_i)\bigr)<\delta.
\end{displaymath}
Using the integral representation in the statement and the norm estimate for
the Bochner integral, we obtain
\begin{displaymath}
 \begin{aligned}
 \sum_{i=1}^n\|u(b_i)-u(a_i)\|_Y
 &\leq\sum_{i=1}^n
 \int_{[a_i,b_i)}\|v(s)\|_Y\,d\mu_g(s)\\
 &=\int_E\|v(s)\|_Y\,d\mu_g(s)<\varepsilon.
 \end{aligned}
\end{displaymath}
This is Definition~\ref{defgabs}; hence $u\in AC_g([a,b];Y)$.
The first half of the proof, now applied to this curve, also gives
$V_u\in AC_g([a,b])$ and \eqref{eq:metric-g-derivative}.
\end{proof}

\begin{rem}[Where the range-space hypothesis enters]
	The two implications in Theorem~\ref{ftc} have a different nature. The
	implication from an integral representation to $g$-absolute continuity is
	valid in every Banach space, whereas the converse implication uses the
	Radon--Nikod\'ym property of the range space.
	
	Assume first that there exists $v\in L_g^1([a,b);Y)$ such that
	\begin{displaymath}
		u(t)=u(a)+\int_{[a,t)}v(s)\,d\mu_g(s),
		\qquad t\in[a,b].
	\end{displaymath}
	For $a\leq c<d\leq b$, additivity of the Bochner integral gives
	\begin{displaymath}
		u(d)-u(c)=\int_{[c,d)}v(s)\,d\mu_g(s).
	\end{displaymath}
	Hence
	\begin{displaymath}
		\|u(d)-u(c)\|_Y
		\leq
		\int_{[c,d)}\|v(s)\|_Y\,d\mu_g(s).
	\end{displaymath}
	Since $\|v(\cdot)\|_Y\in L_g^1([a,b))$, the scalar integral is
	absolutely continuous with respect to $\mu_g$. Thus, for every
	$\varepsilon>0$, there exists $\delta>0$ such that
	\begin{displaymath}
		\mu_g(E)<\delta
	\end{displaymath}
	implies
	\begin{displaymath}
		\int_E\|v(s)\|_Y\,d\mu_g(s)<\varepsilon
	\end{displaymath}
	for every $g$-measurable set $E\subset[a,b)$.
	
	Let $[c_j,d_j)$, $j=1,\ldots,m$, be pairwise disjoint half-open
	intervals satisfying
	\begin{displaymath}
		\sum_{j=1}^{m}\mu_g([c_j,d_j))<\delta.
	\end{displaymath}
	Set
	\begin{displaymath}
		E:=\bigcup_{j=1}^{m}[c_j,d_j).
	\end{displaymath}
	Then
	\begin{displaymath}
		\mu_g(E)=\sum_{j=1}^{m}\mu_g([c_j,d_j))<\delta,
	\end{displaymath}
	and therefore
	\begin{displaymath}
		\begin{aligned}
			\sum_{j=1}^{m}\|u(d_j)-u(c_j)\|_Y
			&\leq
			\sum_{j=1}^{m}\int_{[c_j,d_j)}\|v(s)\|_Y\,d\mu_g(s)\\
			&=
			\int_E\|v(s)\|_Y\,d\mu_g(s)
			<\varepsilon.
		\end{aligned}
	\end{displaymath}
	This is exactly the $g$-absolute continuity condition. The argument uses
	only the norm estimate for the Bochner integral and the absolute continuity
	of a scalar integral. No Radon--Nikod\'ym property of $Y$ is needed in this
	direction.
	
	For the converse implication, assume that $u$ is $g$-absolutely continuous.
	Proposition~\ref{prop:BV-vector-measure} associates with $u$ a countably
	additive $Y$-valued measure $\nu_u$ satisfying
	\begin{displaymath}
		\nu_u([c,d))=u(d)-u(c).
	\end{displaymath}
	The proof of Theorem~\ref{ftc} shows that
	\begin{displaymath}
		|\nu_u|=\mu_{V_u}
	\end{displaymath}
	and
	\begin{displaymath}
		\mu_{V_u}\ll\mu_g.
	\end{displaymath}
	Consequently,
	\begin{displaymath}
		|\nu_u|\ll\mu_g.
	\end{displaymath}
	If $E$ is $g$-measurable and $\mu_g(E)=0$, then
	\begin{displaymath}
		\|\nu_u(E)\|_Y\leq |\nu_u|(E)=0.
	\end{displaymath}
	Thus
	\begin{displaymath}
		\nu_u\ll\mu_g.
	\end{displaymath}
	
	This absolute continuity relation does not, in an arbitrary Banach space,
	ensure that $\nu_u$ has a Bochner-integrable density with values in $Y$.
	The Radon--Nikod\'ym property is used precisely here. By the vector
	Radon--Nikod\'ym theorem, there exists $v\in L_g^1([a,b);Y)$ such that
	\begin{displaymath}
		\nu_u(E)=\int_E v(s)\,d\mu_g(s)
	\end{displaymath}
	for every $g$-measurable set $E\subset[a,b)$. Taking $E=[a,t)$ gives
	\begin{displaymath}
		\begin{aligned}
			u(t)-u(a)
			&=\nu_u([a,t))\\
			&=\int_{[a,t)}v(s)\,d\mu_g(s),
		\end{aligned}
	\end{displaymath}
	and hence
	\begin{displaymath}
		u(t)=u(a)+\int_{[a,t)}v(s)\,d\mu_g(s).
	\end{displaymath}
	
	Therefore, the implication
	\begin{displaymath}
		\text{$u$ has a Bochner integral representation}
		\quad\Longrightarrow\quad
		u\in AC_g([a,b];Y)
	\end{displaymath}
	holds in every Banach space. By contrast, the implication
	\begin{displaymath}
		u\in AC_g([a,b];Y)
		\quad\Longrightarrow\quad
		\text{$u$ has a Bochner integral representation}
	\end{displaymath}
	requires the Radon--Nikod\'ym property in order to represent the vector
	measure $\nu_u$ by a $Y$-valued Bochner density. Thus the theorem separates
	the measure-theoretic statement
	\begin{displaymath}
		|\nu_u|\ll\mu_g
	\end{displaymath}
	from the geometric property of the range space that guarantees the existence
	of a density $d\nu_u/d\mu_g$ with values in $Y$.
\end{rem}

\begin{cor}[Reflexive range spaces]\label{cor:ftc-reflexive}
If $Y$ is reflexive, then the equivalence and the variation identity in
Theorem~\ref{ftc} hold for every $u:[a,b]\to Y$.
\end{cor}

\begin{proof}
Every reflexive Banach space has the Radon--Nikod\'ym property; see
\cite[Ch.~III, Sect.~2]{DiestelUhl1977}.  The conclusion is therefore an
immediate application of Theorem~\ref{ftc}.
\end{proof}

\begin{rem}\label{rem:RNP-necessary-ftc}
	The assumption that $Y$ has the Radon--Nikod\'ym property cannot, in
	general, be omitted. The obstruction already appears for the classical
	clock
	\begin{displaymath}
		g(t)=t.
	\end{displaymath}
	In this case, $\mu_g$ is Lebesgue measure on $[a,b)$ and
	$g$-absolute continuity coincides with the usual absolute continuity of
	$Y$-valued curves.
	
	Recall that a function $u:[a,b]\to Y$ is absolutely continuous if, for every
	$\varepsilon>0$, there exists $\delta>0$ such that, for every finite family
	of pairwise disjoint intervals
	\begin{displaymath}
		(c_1,d_1),\ldots,(c_m,d_m)\subset[a,b]
	\end{displaymath}
	satisfying
	\begin{displaymath}
		\sum_{j=1}^{m}(d_j-c_j)<\delta,
	\end{displaymath}
	one has
	\begin{displaymath}
		\sum_{j=1}^{m}\|u(d_j)-u(c_j)\|_Y<\varepsilon.
	\end{displaymath}
	Every Lipschitz curve is absolutely continuous. However, in a general
	infinite-dimensional Banach space, an absolutely continuous curve, and even
	a Lipschitz curve, need not be differentiable in norm almost everywhere.
	Here differentiability at $t$ means that there exists $u'(t)\in Y$ such that
	\begin{displaymath}
		\lim_{h\to0}
		\left\|
		\frac{u(t+h)-u(t)}{h}-u'(t)
		\right\|_Y
		=0,
	\end{displaymath}
	whenever $t+h\in[a,b]$.
	
	The connection with the Radon--Nikod\'ym property is best expressed in terms
	of vector measures. An absolutely continuous curve $u$ defines a
	$Y$-valued measure $\nu_u$ by
	\begin{displaymath}
		\nu_u([c,d))=u(d)-u(c).
	\end{displaymath}
	The absolute continuity of $u$ implies
	\begin{displaymath}
		\nu_u\ll\mathcal{L}^{1},
	\end{displaymath}
	where $\mathcal{L}^{1}$ denotes Lebesgue measure. If $Y$ has the
	Radon--Nikod\'ym property, the vector-valued Radon--Nikod\'ym theorem provides
	$v\in L^1(a,b;Y)$ such that
	\begin{displaymath}
		\nu_u(E)=\int_E v(s)\,\mathrm{d}s
	\end{displaymath}
	for every Borel set $E\subset[a,b)$. Applying this identity to $[a,t)$ gives
	\begin{displaymath}
		\nu_u([a,t))=u(t)-u(a)=\int_a^t v(s)\,ds,
	\end{displaymath}
	and hence
	\begin{displaymath}
		u(t)=u(a)+\int_a^t v(s)\,\mathrm{d}s.
	\end{displaymath}
	The vector-valued differentiation theorem then yields
	\begin{displaymath}
		u'(t)=v(t)
	\end{displaymath}
	in the norm of $Y$ for almost every $t\in(a,b)$.
	
	Conversely, one of the standard characterizations of the
	Radon--Nikod\'ym property states that $Y$ has this property if and only if
	every Lipschitz mapping from a real interval into $Y$ is differentiable in
	norm almost everywhere. Therefore, if $Y$ fails the Radon--Nikod\'ym
	property, there exists a Lipschitz, and hence absolutely continuous, curve
	$u:[a,b]\to Y$ that is not differentiable in norm on a set of positive
	Lebesgue measure. Such a curve cannot admit a representation
	\begin{displaymath}
		u(t)=u(a)+\int_a^t v(s)\,\mathrm{d}s
	\end{displaymath}
	with $v\in L^1(a,b;Y)$, because such a representation would imply norm
	differentiability almost everywhere.
	
	Thus the range-space assumption in the vector-valued fundamental theorem is
	not merely technical. Already for $g(t)=t$, the validity of the theorem for
	all absolutely continuous $Y$-valued curves is equivalent to the
	Radon--Nikod\'ym property of $Y$; see
	\cite[Ch.~III, Sects.~1--2 and Ch.~VII, Sect.~1]{DiestelUhl1977}.
	Accordingly, the RNP hypothesis is sharp uniformly over the full class of
	derivators treated in this paper, since that class contains the classical
	clock.  This uniform sharpness should not be read as saying that the full RNP
	is necessary for every fixed special clock; for instance, particular purely
	atomic clocks may admit simpler finite- or sequence-valued representations.
\end{rem}

\section{Aubin--Lions compactness for Stieltjes evolution graphs}
\label{sec:aubin-lions}

Let $B_0$ and $B_1$ be Banach spaces and let
$i_{01}:B_0\hookrightarrow B_1$ be a fixed continuous embedding.  Whenever
an intermediate space $B$ is present, we write $i_{0B}:B_0\to B$ and
$i_{B1}:B\to B_1$ for the corresponding embeddings, so that
$i_{01}=i_{B1}\circ i_{0B}$.  We define the evolution space as the graph of the Stieltjes differentiation
relation. The definition is related to the space introduced in
\cite[formula~(26)]{Fran2020}, but retains the density as a component. This is
necessary when the restricted Stieltjes measure has a terminal atom.

\begin{dfn}[Stieltjes evolution graph]\label{def:Stieltjes-evolution-graph}
For $1\leq p_0,p_1\leq\infty$, the \emph{Stieltjes evolution graph} is
\begin{displaymath}
 \begin{aligned}
 \mathbb W_g^{p_0,p_1}([a,b];B_0,B_1)
 :=\bigl\{(u,z):\;&
 u\in L_g^{p_0}([a,b);B_0),\\
 &z\in L_g^{p_1}([a,b);B_1),\\
 &\text{there exists }x\in B_1\text{ such that}\\
 &i_{01}u(t)=x+\int_{[a,t)}z(r)\,d\mu_g(r)\\
 &\text{for }\mu_g\text{-almost every }t\in[a,b)
 \bigr\}.
 \end{aligned}
\end{displaymath}
It is endowed with the norm
\begin{equation}\label{eq:graph-norm}
 \|(u,z)\|_{\mathbb W_g^{p_0,p_1}}
 =\|u\|_{L_g^{p_0}(B_0)}+\|z\|_{L_g^{p_1}(B_1)}.
\end{equation}
We write $\pi_0(u,z)=u$ and $D_g(u,z)=z$. The term ``graph'' refers to
the graph of the differentiation relation between the state and derivative
components.
\end{dfn}

We call $s\in D_g\cap[a,b)$ a \emph{terminal atom} of the restricted
Stieltjes measure
if
\begin{equation}\label{eq:terminal-atom}
 \mu_g((s,b))=0.
\end{equation}
Thus a terminal atom is the last point at which the restricted measure
acquires mass; the interval to its right may still be nonempty in ordinary
time.

Assume that $M:=\mu_g([a,b))>0$. For a fixed pair $(u,z)$ in the graph,
the element $x$ in Definition~\ref{def:Stieltjes-evolution-graph} is unique.
Indeed, the difference of two possible choices is a constant element of
$B_1$ that vanishes $\mu_g$-almost everywhere; since $M>0$, this constant is
zero. We define the \emph{canonical representative} associated with $(u,z)$
by
\begin{equation}\label{eq:canonical-representative-definition}
 \widetilde u(t)
 :=x+\int_{[a,t)}z(r)\,d\mu_g(r),
 \qquad t\in[a,b].
\end{equation}
By Definition~\ref{def:Stieltjes-evolution-graph},
$\widetilde u=i_{01}u$ $\mu_g$-almost everywhere. Since $\mu_g([a,b))<\infty$,
$L_g^{p_1}([a,b);B_1)\subset L_g^1([a,b);B_1)$ for every
$1\leq p_1\leq\infty$. Theorem~\ref{t2.9} therefore gives
\begin{displaymath}
 \widetilde u\in AC_g([a,b];B_1),
 \qquad
 \widetilde u'_g=z
 \quad g\text{-almost everywhere}.
\end{displaymath}
If $M=0$, each Bochner factor in the definition consists of a single
equivalence class. Hence the graph is the singleton containing the zero pair.
The element $x$ in the almost-everywhere representation is then not unique,
but this ambiguity has no effect on the graph element or on its norm; no
canonical pointwise lift is needed in this degenerate case.

\begin{exa}[A terminal atom and nonuniqueness of the derivative]
\label{ex:terminal-atom-nonuniqueness}
Let $Y=\mathbb R$, let $[a,b]=[0,2]$, and fix $\alpha>0$.  Define the
left-continuous derivator $g:[0,2]\to\mathbb R$ by
\begin{displaymath}
g(t)=
\begin{cases}
0, & 0\leq t\leq1,\\
\alpha, & 1<t\leq2.
\end{cases}
\end{displaymath}
Then
\begin{displaymath}
\mu_g=\alpha\delta_{\{1\}},
\qquad
\mu_g((1,2))=0,
\end{displaymath}
so $1$ is a terminal atom. For every $c\in\mathbb R$, choose the
representative
\begin{displaymath}
 z_c=c\,\mathbf1_{\{1\}}
 \in L_g^p([0,2);\mathbb R)
\end{displaymath}
and let the initial value be $x=0$. The associated canonical representative
is
\begin{displaymath}
\widetilde u_c(t)
=
\int_{[0,t)}z_c(r)\,d\mu_g(r)
=
\begin{cases}
0, & 0\leq t\leq1,\\
\alpha c, & 1<t\leq2.
\end{cases}
\end{displaymath}
In particular, $\widetilde u_c(1)=0$ for every $c$.  Hence all these
representatives determine the same state class, namely the zero element of
$L_g^p([0,2);\mathbb R)$.  For $1\leq p<\infty$ and $c,d\in\mathbb R$,
\begin{displaymath}
\|\widetilde u_c-\widetilde u_d\|_{L_g^p([0,2);\mathbb R)}^p
=
\alpha|\widetilde u_c(1)-\widetilde u_d(1)|^p
=
0,
\end{displaymath}
and the same conclusion holds for $p=\infty$.

The derivatives are nevertheless different.  Indeed,
\begin{displaymath}
\Delta^+g(1)=\alpha,
\qquad
\Delta^+\widetilde u_c(1)=\alpha c,
\qquad
(\widetilde u_c)'_g(1)=c.
\end{displaymath}
Thus, if $c\ne d$ and $1\leq p<\infty$,
\begin{displaymath}
\|(\widetilde u_c)'_g-(\widetilde u_d)'_g
\|_{L_g^p([0,2);\mathbb R)}^p
=
\alpha|c-d|^p
>
0;
\end{displaymath}
for $p=\infty$, the corresponding norm is $|c-d|$.  We therefore have the
same state class but distinct derivative classes.

The reason is that the Bochner state class records the left-continuous value
at the atom, whereas the derivative at that atom determines the value
immediately after the jump.  Since the interval to the right of the atom has
zero $\mu_g$-measure, this post-jump value is invisible to the state class.
Consequently, the derivative cannot in general be recovered from the state
class alone, which motivates retaining the pair $(u,z)$ in
Definition~\ref{def:Stieltjes-evolution-graph}. One could restore uniqueness
by requiring the representative to remain equal to its value at the terminal
atom on the interval to its right. In the present example, however, this would
impose $\alpha c=0$ and hence $c=0$; it would exclude every nontrivial
terminal impulse rather than represent it. When the derivative is known to be
uniquely determined by the state---as in ordinary time and in the atomic
measure of Corollary~\ref{cor:weighted-discrete}---the graph can instead be
identified with a space of state functions.
\end{exa}

The following estimate controls both the size and the temporal variation of
the canonical representative.

\begin{pro}[Canonical representative and temporal control]
\label{prop:canonical-temporal-control}
	Assume that $M:=\mu_g([a,b))>0$ and denote by $c_{01}$ the norm of the
	embedding $B_0\hookrightarrow B_1$.  If
	$\boldsymbol u=(u,z)\in
	\mathbb W_g^{p_0,p_1}([a,b];B_0,B_1)$, where
	$1\leq p_0,p_1\leq\infty$, then its canonical representative
	$\widetilde u:[a,b]\to B_1$ satisfies
	\begin{equation}\label{eq:canonical-representative-bound}
		\sup_{t\in[a,b]}\|\widetilde u(t)\|_{B_1}
		\leq
		c_{01}M^{-1/p_0}\|u\|_{L^{p_0}_g(B_0)}
		+M^{1-1/p_1}\|z\|_{L^{p_1}_g(B_1)},
	\end{equation}
	with the usual conventions when an exponent is infinite. Moreover, for
	$a\leq s<t\leq b$,
	\begin{equation}\label{eq:canonical-temporal-modulus}
		\|\widetilde u(t)-\widetilde u(s)\|_{B_1}
		\leq
		\int_{[s,t)}\|z(r)\|_{B_1}\,d\mu_g(r).
	\end{equation}
	If $p_1>1$, Hölder's inequality further bounds the right-hand side by
	\begin{displaymath}
		\|z\|_{L^{p_1}_g(B_1)}
		\mu_g([s,t))^{1-1/p_1}.
	\end{displaymath}
\end{pro}

\begin{proof}
	Choose a strongly measurable representative of the Bochner class $u$ and
let
\begin{displaymath}
 E_0:=\{t\in[a,b):\widetilde u(t)=i_{01}u(t)\}.
\end{displaymath}
Then $\mu_g([a,b)\setminus E_0)=0$. We show that there is a point of
$E_0$ at which the required norm bound holds.

Suppose first that $p_0<\infty$ and set
\begin{displaymath}
 c:=M^{-1}\|u\|_{L_g^{p_0}(B_0)}^{p_0}.
\end{displaymath}
The measurable set
\begin{displaymath}
 E_1:=\{t\in[a,b):\|u(t)\|_{B_0}^{p_0}\leq c\}
\end{displaymath}
cannot be $\mu_g$-null. Otherwise
$\|u(t)\|_{B_0}^{p_0}>c$ for almost every $t$, and the nonnegative
function $\|u(\cdot)\|_{B_0}^{p_0}-c$ would be positive almost
everywhere on a set of positive finite measure. Its integral would
therefore be strictly positive, contradicting
\begin{displaymath}
 \int_{[a,b)}\bigl(\|u(t)\|_{B_0}^{p_0}-c\bigr)\,d\mu_g(t)=0.
\end{displaymath}
Since $E_0$ has full measure and $E_1$ has positive measure,
$E_0\cap E_1$ is nonempty. Choose $\tau$ in this intersection.

If $p_0=\infty$, the inequality
\begin{displaymath}
 \|u(t)\|_{B_0}\leq\|u\|_{L_g^\infty(B_0)}
\end{displaymath}
holds almost everywhere. Intersecting this full-measure set with $E_0$
again gives a permissible point $\tau$. In either case,
\begin{displaymath}
 \widetilde u(\tau)=i_{01}u(\tau),
 \qquad
 \|u(\tau)\|_{B_0}
 \leq M^{-1/p_0}\|u\|_{L^{p_0}_g(B_0)},
\end{displaymath}
with the convention $M^{-1/\infty}=1$.

From the integral representation, independently of the order of $t$ and
	$\tau$,
	\begin{displaymath}
		\|\widetilde u(t)-\widetilde u(\tau)\|_{B_1}
		\leq\int_{[a,b)}\|z(r)\|_{B_1}\,d\mu_g(r).
	\end{displaymath}
	Hölder's inequality bounds the right-hand side by
	$M^{1-1/p_1}\|z\|_{L^{p_1}_g(B_1)}$.  Since
	$\|\widetilde u(\tau)\|_{B_1}\leq c_{01}\|u(\tau)\|_{B_0}$, this proves
	\eqref{eq:canonical-representative-bound}.  Applying the same argument
	to
	\begin{displaymath}
		\widetilde u(t)-\widetilde u(s)=\int_{[s,t)}z(r)\,d\mu_g(r)
	\end{displaymath}
	proves \eqref{eq:canonical-temporal-modulus}; its final assertion is
	another application of Hölder's inequality.
\end{proof}

\begin{cor}[Canonical lift and bounded evaluations]
\label{cor:canonical-lift}
Assume that $M=\mu_g([a,b))>0$. The canonical lift
\begin{displaymath}
 \mathcal J_g:\mathbb W_g^{p_0,p_1}([a,b];B_0,B_1)
 \longrightarrow\mathcal{BC}_g([a,b];B_1),
 \qquad \mathcal J_g(u,z)=\widetilde u,
\end{displaymath}
is linear and bounded, and its range is contained in
$AC_g([a,b];B_1)$. In particular, evaluation of the canonical representative
at any fixed $t\in[a,b]$ is a bounded linear map into $B_1$. If
$s\in D_g\cap[a,b)$, evaluation of the state class at $s$ is a bounded linear
map from $L_g^{p_0}([a,b);B_0)$ to $B_0$, with
\begin{equation}\label{eq:atom-evaluation-general}
 \|u(s)\|_{B_0}
 \leq \mu_g(\{s\})^{-1/p_0}
       \|u\|_{L_g^{p_0}(B_0)}.
\end{equation}
\end{cor}

\begin{proof}
Linearity of $\mathcal J_g$ follows from uniqueness of the canonical
representative when $M>0$, and its boundedness is precisely
\eqref{eq:canonical-representative-bound}. The range inclusion follows from
\eqref{eq:canonical-representative-definition}, the finite-measure embedding
$L_g^{p_1}(B_1)\subset L_g^1(B_1)$ and Theorem~\ref{t2.9}. The norm of point
evaluation on $\mathcal{BC}_g([a,b];B_1)$ is at most one, which proves the
second assertion. Finally, the value of a Bochner class at an atom is well
defined,
and, if $p_0<\infty$,
\begin{displaymath}
 \mu_g(\{s\})\|u(s)\|_{B_0}^{p_0}
 \leq\int_{[a,b)}\|u(r)\|_{B_0}^{p_0}\,d\mu_g(r).
\end{displaymath}
Taking the $p_0$-th root gives \eqref{eq:atom-evaluation-general}.  If
$p_0=\infty$, the value at the positive-mass atom cannot exceed the essential
supremum, and the same formula holds with the convention $1/\infty=0$.
\end{proof}

\begin{pro}[Banach and reflexivity properties]\label{prop:W-structure}
	Let $1\leq p_0,p_1\leq\infty$. Then
	$\mathbb W_g^{p_0,p_1}([a,b];B_0,B_1)$ is a Banach space.  If $B_0$ and
	$B_1$ are reflexive and $1<p_0,p_1<\infty$, the graph is reflexive.
\end{pro}

\begin{proof}
	Put $M=\mu_g([a,b))$.  If $M=0$, both factors in
	Definition~\ref{def:Stieltjes-evolution-graph} are zero spaces and there
	is nothing to prove.  Assume $M>0$.

	Let $\boldsymbol u_n=(u_n,z_n)$ be a Cauchy sequence in the graph.  The
	completeness of Bochner spaces gives functions
	$u\in L^{p_0}_g([a,b);B_0)$ and
	$z\in L^{p_1}_g([a,b);B_1)$ such that
	\begin{displaymath}
		u_n\longrightarrow u\quad\text{in }L^{p_0}_g(B_0),
		\qquad
		z_n\longrightarrow z\quad\text{in }L^{p_1}_g(B_1).
	\end{displaymath}
	Let $U_n$ be the canonical representative associated with
	$(u_n,z_n)$.  The difference pair
	$(u_n-u_m,z_n-z_m)$ also belongs to the graph and has representative
	$U_n-U_m$.  Applying \eqref{eq:canonical-representative-bound} to this
	pair shows that $(U_n)$ is Cauchy in
	$\mathcal{BC}_g([a,b];B_1)$ with the uniform norm.  The uniform limit of
	$g$-continuous functions is $g$-continuous.  Denote this limit by
	$U:[a,b]\to B_1$.

	Since $i_{01}$ is continuous, $i_{01}u_n\to i_{01}u$ in
	$L^{p_0}_g(B_1)$.  On the other hand,
	$U_n=i_{01}u_n$ almost everywhere and uniform convergence gives
	$U_n\to U$ in $L^{p_0}_g(B_1)$.  Uniqueness of the limit in this Bochner
	space yields
	\begin{equation}\label{eq:identify-W-limit}
		U=i_{01}u\qquad\mu_g\text{-almost everywhere}.
	\end{equation}

	Because $M<\infty$, Hölder's inequality implies $z_n\to z$ in
	$L^1_g(B_1)$.  Hence
	\begin{displaymath}
	 \sup_{t\in[a,b]}
	 \left\|\int_{[a,t)}(z_n(r)-z(r))\,d\mu_g(r)\right\|_{B_1}
	 \leq\|z_n-z\|_{L^1_g(B_1)}\longrightarrow0.
	\end{displaymath}
	Passing to the limit in
	\begin{displaymath}
		U_n(t)=U_n(a)+\int_{[a,t)}z_n(r)\,d\mu_g(r)
	\end{displaymath}
	therefore gives, for every $t\in[a,b]$,
	\begin{displaymath}
		U(t)=U(a)+\int_{[a,t)}z(r)\,d\mu_g(r).
	\end{displaymath}
	Together with \eqref{eq:identify-W-limit}, this proves that $(u,z)$ belongs
	to the graph.  Convergence in the two Bochner factors is the same as
	convergence in the norm \eqref{eq:graph-norm}.  The graph is therefore
	complete.

	Now suppose that $B_0$ and $B_1$ are reflexive and that both exponents
	belong to $(1,\infty)$.  Then the Bochner spaces
	$L_g^{p_0}(B_0)$ and $L_g^{p_1}(B_1)$ are reflexive.  Their finite product
	is reflexive as well.  The first part of the proof shows that
	$\mathbb W_g^{p_0,p_1}$ is a closed linear subspace of that product, so it
	is reflexive.
\end{proof}

\begin{lem}[Compact maps improve weak convergence]\label{lem:compact-weak-strong}
Let $X$ and $Y$ be Banach spaces and let $T:X\to Y$ be a compact linear operator.  If
$x_n\rightharpoonup x$ in $X$, then $Tx_n\to Tx$ in $Y$.
\end{lem}

\begin{proof}
The sequence $(Tx_n)$ is relatively compact.  Every norm-convergent
subsequence has a limit $y\in Y$, while bounded linearity gives
$Tx_n\rightharpoonup Tx$ in $Y$.  Hence $y=Tx$.  Since every subsequence has
a further subsequence converging in norm to $Tx$, the whole sequence converges
to $Tx$ in $Y$.
\end{proof}

\begin{thm}[Aubin--Lions compactness for Stieltjes evolution graphs]
\label{thm:AubinLions}
	Let $g:\mathbb R\to\mathbb R$ be a derivator and let $B_0$, $B$ and
	$B_1$ be Banach spaces such that
	\begin{itemize}
		\item $B_0$ and $B_1$ are reflexive;
		\item $B_0\Subset B$;
		\item $B\hookrightarrow B_1$ continuously.
	\end{itemize}
	If $1<p_0,p_1<\infty$, then
	the state projection
	\begin{displaymath}
	 \pi_0:\mathbb W_g^{p_0,p_1}([a,b];B_0,B_1)
	 \longrightarrow L^{p_0}_g([a,b);B),
	 \qquad \pi_0(u,z)=u,
	\end{displaymath}
	is compact.  Equivalently, from every bounded sequence of admissible
	state--derivative pairs one can extract a subsequence whose state
	components converge strongly in $L_g^{p_0}([a,b);B)$.
\end{thm}

\begin{proof}
	Write $\mu=\mu_g$ and $M=\mu([a,b))$.  If $M=0$, the target space is the
	zero space, and the result is immediate.  We assume from now on that
	$M>0$.

	Let $\boldsymbol u_n=(u_n,z_n)$ be bounded in
	$\mathbb W:=\mathbb W_g^{p_0,p_1}([a,b];B_0,B_1)$.  By
	Proposition~\ref{prop:W-structure}, this graph is reflexive.  Passing to a
	subsequence, there is $\boldsymbol u=(u,z)\in\mathbb W$ such that
	\begin{displaymath}
	 \boldsymbol u_n\rightharpoonup\boldsymbol u
	 \quad\text{weakly in }\mathbb W.
	\end{displaymath}
	Set
	\begin{displaymath}
	 w_n=u_n-u,
	 \qquad q_n=z_n-z,
	 \qquad\boldsymbol w_n=(w_n,q_n).
	\end{displaymath}
	The sequence $\boldsymbol w_n$ converges weakly to zero in the graph.
	The two coordinate projections of the graph into its Bochner factors are
	bounded linear maps; consequently, $w_n$ and $q_n$ converge weakly to zero
	in their respective Bochner spaces.  There is also a constant $K>0$ such that
	\begin{equation}\label{eq:W-bounds-Aubin-Lions}
	 \|w_n\|_{L_g^{p_0}(B_0)}+\|q_n\|_{L_g^{p_1}(B_1)}\leq K
	 \qquad(n\in\mathbb N).
	\end{equation}

	\smallskip
	\noindent\emph{Uniform control of the canonical representatives.}
	Let $W_n:[a,b]\to B_1$ be the canonical representative associated with
	$\boldsymbol w_n$.  Proposition~\ref{prop:canonical-temporal-control}
	and \eqref{eq:W-bounds-Aubin-Lions} give a constant $K_1$, independent of
	$n$, such that
	\begin{equation}\label{eq:uniform-B1-Aubin-Lions}
	 \sup_{n\in\mathbb N}\sup_{t\in[a,b]}\|W_n(t)\|_{B_1}\leq K_1.
	\end{equation}
	If $a\leq r<t\leq b$ and $\alpha=1-1/p_1>0$, the same proposition and
	Hölder's inequality give
	\begin{equation}\label{eq:time-control-Aubin-Lions}
	 \|W_n(t)-W_n(r)\|_{B_1}
	 \leq K\mu([r,t))^\alpha.
	\end{equation}

	\smallskip
	\noindent\emph{Pointwise convergence at the atoms.}
	We now prove pointwise convergence of these representatives on a set of
	full $\mu$-measure.  Let first $s\in D_g\cap[a,b)$ and put
	$\delta_s=\mu(\{s\})>0$.  Evaluation at $s$ is a bounded linear map from
	$L_g^{p_0}(B_0)$ to $B_0$, because
	\begin{equation}\label{eq:atomic-evaluation-bound}
	 \|v(s)\|_{B_0}^{p_0}\delta_s
	 \leq\int_{[a,b)}\|v(r)\|_{B_0}^{p_0}\,d\mu(r)
	\end{equation}
	for every $v\in L_g^{p_0}(B_0)$.  The value at an atom is well defined for
	a Bochner equivalence class.  Moreover,
	$W_n(s)=i_{01}w_n(s)$, since the canonical representative and the state
	component agree almost everywhere and $\{s\}$ has positive measure.
	Thus $w_n(s)\rightharpoonup0$ in $B_0$.  The composition
	$i_{01}=i_{B1}\circ i_{0B}:B_0\to B_1$ is compact.
	Lemma~\ref{lem:compact-weak-strong} therefore gives
	$w_n(s)\to0$ in $B_1$.
	Consequently,
	\begin{equation}\label{eq:atomic-pointwise-convergence}
	 W_n(s)\longrightarrow0\quad\text{in }B_1,
	 \qquad s\in D_g\cap[a,b).
	\end{equation}

	\smallskip
	\noindent\emph{Pointwise convergence at nonatomic support points.}
	Let now $s\in\operatorname{supp}\mu\cap[a,b)$ and suppose that
	$\mu(\{s\})=0$, where the support is taken in the relative topology of
	$[a,b)$.  Choose a decreasing sequence $\rho_k\downarrow0$ and put
	\begin{displaymath}
	 I_k=[\max\{a,s-\rho_k\},\min\{b,s+\rho_k\}).
	\end{displaymath}
	The intervals are nested, contain $s$, and have intersection $\{s\}$.
	Since $s$ belongs to
	the support of $\mu$, $m_k:=\mu(I_k)>0$.  Continuity from above for the
	finite measure $\mu$ gives
	\begin{displaymath}
	 m_k\longrightarrow\mu(\{s\})=0.
	\end{displaymath}
	For fixed $k$, consider the Bochner average
	\begin{displaymath}
	 A_{n,k}=\frac1{m_k}\int_{I_k}w_n(r)\,d\mu(r)\in B_0.
	\end{displaymath}
	This integral is well defined for a strongly measurable representative
	of the Bochner class; changing the representative does not change the
	average.  The averaging operator is bounded, since Hölder's inequality gives
	\begin{displaymath}
	 \|A_{n,k}\|_{B_0}
	 \leq m_k^{-1/p_0}\|w_n\|_{L_g^{p_0}(I_k;B_0)}.
	\end{displaymath}
	The averaging map
	$L_g^{p_0}([a,b);B_0)\to B_0$ is therefore bounded.  Since
	$w_n\rightharpoonup0$ in the domain, it follows that
	$A_{n,k}\rightharpoonup0$ in $B_0$ as $n\to\infty$.  Applying
	Lemma~\ref{lem:compact-weak-strong} to $i_{01}:B_0\to B_1$ yields
	\begin{equation}\label{eq:average-convergence}
	 i_{01}A_{n,k}\longrightarrow0\quad\text{in }B_1
	 \quad\text{for every fixed }k.
	\end{equation}

	The equality $W_n=i_{01}w_n$ holds $\mu$-almost everywhere.  Hence, as an
	identity in $B_1$,
	\begin{displaymath}
	 W_n(s)-i_{01}A_{n,k}
	 =\frac1{m_k}\int_{I_k}\bigl(W_n(s)-W_n(r)\bigr)\,d\mu(r).
	\end{displaymath}
	For $r\in I_k$, the half-open interval with endpoints $r$ and $s$ is
	contained in $I_k$.  Estimate \eqref{eq:time-control-Aubin-Lions}, used in
	the appropriate order of the endpoints, therefore gives
	\begin{equation}\label{eq:value-average-estimate}
	 \|W_n(s)-i_{01}A_{n,k}\|_{B_1}
	 \leq\frac1{m_k}\int_{I_k}K m_k^\alpha\,d\mu
	 =K m_k^\alpha.
	\end{equation}
	Combining \eqref{eq:average-convergence} and
	\eqref{eq:value-average-estimate}, we obtain
	\begin{displaymath}
	 \limsup_{n\to\infty}\|W_n(s)\|_{B_1}\leq K m_k^\alpha.
	\end{displaymath}
	Letting $k\to\infty$ shows that $W_n(s)\to0$ in $B_1$.  The complement
	of $\operatorname{supp}\mu$ is $\mu$-null.  To recall the short argument,
	every point of that complement has an open neighbourhood of zero measure;
	a countable base of the real line reduces the union of all these
	neighbourhoods to a countable union of null sets.  Together with the atomic
	case \eqref{eq:atomic-pointwise-convergence}, this proves
	\begin{equation}\label{eq:pointwise-B1-Aubin-Lions}
	 W_n(s)\longrightarrow0\quad\text{in }B_1
	 \quad\text{for $\mu$-almost every }s\in[a,b).
	\end{equation}

	\smallskip
	\noindent\emph{Strong convergence in the weak ambient space.}
	By \eqref{eq:uniform-B1-Aubin-Lions},
	$\|W_n(s)\|_{B_1}^{p_0}\leq K_1^{p_0}$, and this constant is integrable
	because $M<\infty$.  The Bochner state $w_n$ agrees almost everywhere with $W_n$. Hence
	$\|W_n(\cdot)\|_{B_1}^{p_0}$ agrees almost everywhere with the measurable
	function $\|w_n(\cdot)\|_{B_1}^{p_0}$ and is measurable with respect to
	the completed measure. The scalar dominated convergence theorem now gives
	\begin{equation}\label{eq:strong-B1-Aubin-Lions}
	 w_n\longrightarrow0
	 \quad\text{strongly in }L_g^{p_0}([a,b);B_1).
	\end{equation}

	\smallskip
	\noindent\emph{Ehrling's interpolation step.}
	It remains to pass from $B_1$ to $B$.  By Ehrling's lemma,
	Lemma~\ref{l7.6}, for every $\varepsilon>0$ there is
	$C_\varepsilon>0$ such that
	\begin{displaymath}
	 \|x\|_B^{p_0}
	 \leq\varepsilon\|x\|_{B_0}^{p_0}
	 +C_\varepsilon\|x\|_{B_1}^{p_0},
	 \qquad x\in B_0.
	\end{displaymath}
	This inequality holds for $w_n(s)$ at almost every $s$.  Integrating with
	respect to $\mu$ gives
	\begin{displaymath}
	 \|w_n\|_{L_g^{p_0}(B)}^{p_0}
	 \leq\varepsilon\|w_n\|_{L_g^{p_0}(B_0)}^{p_0}
	 +C_\varepsilon\|w_n\|_{L_g^{p_0}(B_1)}^{p_0}.
	\end{displaymath}
	Taking the upper limit as $n\to\infty$ and using
	\eqref{eq:W-bounds-Aubin-Lions} and
	\eqref{eq:strong-B1-Aubin-Lions}, we find
	\begin{displaymath}
	 \limsup_{n\to\infty}\|w_n\|_{L_g^{p_0}(B)}^{p_0}
	 \leq\varepsilon K^{p_0}.
	\end{displaymath}
	Since $\varepsilon$ is arbitrary, $w_n\to0$ in $L_g^{p_0}(B)$.  Thus the
	state components of the selected subsequence converge strongly to $u$,
	which proves compactness of $\pi_0$.
\end{proof}

\begin{proof}[Alternative proof via finite partitions of the $g$-range]
Write $\mu=\mu_g|_{[a,b)}$ and $M=\mu([a,b))=g(b)-g(a)$. If $M=0$, the target space is the zero space and there is nothing to prove. Assume $M>0$, and let $\mathcal F$ be a bounded subset of the evolution graph. The empty family requires no argument, so assume that $\mathcal F$ is nonempty. Choose finite constants $K_0,K_z\ge0$ such that

\begin{displaymath}
\|u\|_{L_g^{p_0}(B_0)}\le K_0,
\qquad
\|z\|_{L_g^{p_1}(B_1)}\le K_z,
\qquad (u,z)\in\mathcal F.
\end{displaymath}

For each pair, let $U:[a,b]\to B_1$ be its canonical representative. Put $\beta=1-1/p_1>0$. The integral representation and Hölder's inequality give, whenever $a\le r<t\le b$,

\begin{displaymath}
\|U(t)-U(r)\|_{B_1}
\le \int_{[r,t)}\|z(s)\|_{B_1}\,d\mu(s)
\le K_z\mu([r,t))^\beta
=K_z\bigl(g(t)-g(r)\bigr)^\beta.
\end{displaymath}

Fix $\delta>0$. Let $N=\lceil M/\delta\rceil$, $h=M/N\le\delta$, and $y_j=g(a)+jh$ for $j=0,\ldots,N$. Partition the whole interval of clock values by

\begin{displaymath}
J_j=[y_{j-1},y_j)\quad(1\le j<N),
\qquad J_N=[y_{N-1},y_N].
\end{displaymath}

When $N=1$, only the last interval is used. Define

\begin{displaymath}
E_j=\{t\in[a,b):g(t)\in J_j\}.
\end{displaymath}

The monotone function $g$ is Borel measurable, so each $E_j$ is Borel. These sets are pairwise disjoint and cover $[a,b)$. The last interval includes $g(b)$, which may already be attained before $b$. Each nonempty cell is order-convex: if two of its points enclose a third point, monotonicity places the clock value of that third point in the same interval $J_j$.

For any $r,t\in E_j$, after ordering the endpoints,

\begin{displaymath}
\mu([\min(r,t),\max(r,t)))=|g(t)-g(r)|\le h\le\delta.
\end{displaymath}

Consequently,

\begin{displaymath}
\|U(t)-U(r)\|_{B_1}\le K_z\delta^\beta,
\qquad r,t\in E_j.
\end{displaymath}

This estimate does not assert that $\mu(E_j)\le\delta$. In particular, an atom in a cell may have mass larger than $\delta$. Its right jump then prevents a later point lying beyond that jump from belonging to the same cell; when the atom is the upper endpoint of an increment, its mass is excluded by the convention $[r,t)$. Thus the preceding estimate remains valid. Plateaux and points on boundaries of clock intervals also cause no double counting, because every clock value is assigned to exactly one $J_j$.

Let $S=\{j:\mu(E_j)>0\}$ and write $m_j=\mu(E_j)$ for $j\in S$. The discarded cells form a null set, and $\sum_{j\in S}m_j=M$. Define the Bochner averages in $B_0$ by

\begin{displaymath}
A_j(u)=\frac1{m_j}\int_{E_j}u(r)\,d\mu(r),
\qquad j\in S.
\end{displaymath}

These integrals exist because the states are strongly measurable and belong to $L_g^{p_0}(B_0)$ on a finite measure space. Their values are independent of the representative of the state class. Hölder's inequality yields

\begin{displaymath}
\|A_j(u)\|_{B_0}
\le m_j^{-1/p_0}\|u\|_{L_g^{p_0}(E_j;B_0)}
\le m_j^{-1/p_0}K_0.
\end{displaymath}

Thus, for each fixed partition and each fixed retained cell, the family of averages is bounded in $B_0$. Set

\begin{displaymath}
P_\delta u=\sum_{j\in S}A_j(u)\mathbf1_{E_j},
\end{displaymath}

with value zero on the discarded cells. This is a $B_0$-valued simple function. Since $i_{01}$ is bounded linear, it commutes with the Bochner integral. Using $U=i_{01}u$ almost everywhere, for $t\in E_j$ we obtain in $B_1$

\begin{displaymath}
U(t)-i_{01}A_j(u)
=\frac1{m_j}\int_{E_j}\bigl(U(t)-U(r)\bigr)\,d\mu(r).
\end{displaymath}

The oscillation estimate therefore gives

\begin{displaymath}
\|U(t)-i_{01}A_j(u)\|_{B_1}\le K_z\delta^\beta.
\end{displaymath}

Passing to the state classes, raising to the power $p_0$, and summing over the retained cells, we find

\begin{displaymath}
\|i_{01}(u-P_\delta u)\|_{L_g^{p_0}(B_1)}^{p_0}
\le\sum_{j\in S}m_jK_z^{p_0}\delta^{\beta p_0}
=M K_z^{p_0}\delta^{\beta p_0}.
\end{displaymath}

In particular,

\begin{displaymath}
\|i_{01}(u-P_\delta u)\|_{L_g^{p_0}(B_1)}
\le K_zM^{1/p_0}\delta^\beta
\qquad ((u,z)\in\mathcal F).
\end{displaymath}

For a fixed partition, the inclusion $i_{01}:B_0\to B_1$ is compact, as the composition of the compact inclusion into $B$ and the continuous inclusion into $B_1$. Hence

\begin{displaymath}
C_j=\overline{\{i_{01}A_j(u):(u,z)\in\mathcal F\}}^{B_1}
\end{displaymath}

is compact for each $j\in S$. There are only finitely many retained cells, so $\prod_{j\in S}C_j$ is compact. Explicitly, successive subsequence extractions in the finitely many coordinates give a subsequence convergent in every coordinate, and hence in the finite product.

The map

\begin{displaymath}
\Phi:\prod_{j\in S}C_j\longrightarrow L_g^{p_0}(B_1),
\qquad \Phi((v_j))=\sum_{j\in S}v_j\mathbf1_{E_j},
\end{displaymath}

is continuous, since

\begin{displaymath}
\|\Phi((v_j))-\Phi((w_j))\|_{L_g^{p_0}(B_1)}^{p_0}
=\sum_{j\in S}m_j\|v_j-w_j\|_{B_1}^{p_0}.
\end{displaymath}

Its image is compact and contains every $i_{01}P_\delta u$. Thus the approximants form a relatively compact family in $L_g^{p_0}(B_1)$.

To prove total boundedness of the original state family, let $\eta>0$. Choose $\delta$ so that $K_zM^{1/p_0}\delta^\beta<\eta/3$; if $K_z=0$, any positive $\delta$ is suitable. The compact set $\Phi(\prod_{j\in S}C_j)$ has a finite $\eta/3$-net. Every state is within $\eta/3$ of its approximant, which is within $\eta/3$ of a point of that net. Hence the same finite set is an $\eta$-net for the states in $L_g^{p_0}(B_1)$. This proves total boundedness. Since $B_1$ is Banach, $L_g^{p_0}(B_1)$ is complete, and the state family is relatively compact there.

Take any sequence of pairs in $\mathcal F$ and extract a subsequence whose states are Cauchy in $L_g^{p_0}(B_1)$. The integrated form of Ehrling's inequality gives, for every $\varepsilon>0$,

\begin{displaymath}
\begin{aligned}
\|u_n-u_m\|_{L_g^{p_0}(B)}^{p_0}
&\le\varepsilon\|u_n-u_m\|_{L_g^{p_0}(B_0)}^{p_0}
+C_\varepsilon\|i_{01}(u_n-u_m)\|_{L_g^{p_0}(B_1)}^{p_0}\\
&\le\varepsilon(2K_0)^{p_0}
+C_\varepsilon\|i_{01}(u_n-u_m)\|_{L_g^{p_0}(B_1)}^{p_0}.
\end{aligned}
\end{displaymath}

First keep $\varepsilon$ fixed and let $n,m\to\infty$; then let $\varepsilon\downarrow0$. The subsequence is Cauchy in $L_g^{p_0}(B)$, which is complete because $B$ is Banach. Its states therefore converge strongly in that space. Since the sequence was arbitrary, the state projection is compact.
\end{proof}

The alternative argument appears to use fewer structural hypotheses; endpoint and generalized formulations are not pursued here.

\begin{rem}[Relation with the classical proof]
	\label{rem:no-reparametrization}
	When
	\begin{displaymath}
		g(t)=t,
	\end{displaymath}
	the Lebesgue--Stieltjes measure associated with $g$ is ordinary
	Lebesgue measure on $[a,b)$. In particular,
	\begin{displaymath}
		\mu_g([s,t))=t-s
	\end{displaymath}
	for every $a\leq s<t\leq b$, the support of $\mu_g$ is the whole
	interval, and $\mu_g$ has no atoms. The estimate
	\eqref{eq:time-control-Aubin-Lions} then becomes
	\begin{displaymath}
		\|W_n(t)-W_n(s)\|_{B_1}
		\leq
		K(t-s)^{1-1/p_1}.
	\end{displaymath}
	Thus small intervals in the ordinary time variable produce uniformly
	small increments in $B_1$. Moreover, the local averages appearing in
	the proof are the usual averages with respect to Lebesgue measure. The
	argument therefore reduces to the standard temporal compactness
	mechanism used in the classical Aubin--Lions theorem.
	
	For a general nondecreasing left-continuous derivator $g$, the effective
	size of a time interval is not its Euclidean length but its clock mass
	\begin{displaymath}
		\mu_g([s,t))=g(t)-g(s).
	\end{displaymath}
	Accordingly, estimate \eqref{eq:time-control-Aubin-Lions} controls the
	increment of a trajectory through the amount of clock mass accumulated
	between $s$ and $t$, rather than through $t-s$.
	
	It is tempting to introduce a new time variable
	\begin{displaymath}
		\tau=g(t)
	\end{displaymath}
	and try to reduce the problem to the classical one. Such a direct
	reparametrization is available under stronger assumptions, for example
	when $g$ is continuous and strictly increasing. In the present generality,
	however, it does not provide a uniform reduction of the proof.
	
	Indeed, if $g$ is constant on a nontrivial interval, then all points of
	that interval have the same clock value. Hence $g$ is not injective and
	has no ordinary inverse. The plateau has zero $\mu_g$-mass in its
	interior, so it is invisible to the corresponding Bochner spaces, but it
	cannot be represented by an invertible change of the physical-time
	variable.
	
	If $g$ has a right jump at $t$, then
	\begin{displaymath}
		\mu_g(\{t\})
		=
		\Delta^+g(t)
		=
		g(t^+)-g(t)
		>0.
	\end{displaymath}
	Thus a single physical-time point carries a positive amount of clock
	mass. The map $t\mapsto g(t)$ skips the clock interval between $g(t)$
	and $g(t^+)$. A generalized inverse may encode this mass, but only by
	collapsing an interval of the new time variable onto the single point
	$t$. This is not an ordinary one-to-one reparametrization and it does
	not eliminate the need to treat atoms separately. In the proof of
	Theorem~\ref{thm:AubinLions}, this atomic contribution is handled
	directly by bounded point evaluation and the compact embedding from
	$B_0$ into $B_1$.
	
	A further obstruction is produced by a singular continuous component of
	$\mu_g$. In general, there need not exist a function
	\begin{displaymath}
		\gamma\in L^1(a,b)
	\end{displaymath}
	such that
	\begin{displaymath}
		d\mu_g(t)=\gamma(t)\,dt.
	\end{displaymath}
	Consequently, one cannot in general rewrite every Stieltjes integral as
	an ordinary weighted Lebesgue integral. In particular, replacing
	$d\mu_g$ by $g'(t)\,dt$ would lose the singular continuous part: it may
	happen that $g'(t)=0$ for Lebesgue-almost every $t$ while the associated
	Lebesgue--Stieltjes measure is nonzero.
	
	For these reasons, the proof is formulated directly on the finite
	measure space
	\begin{displaymath}
		([a,b),\mathcal B([a,b)),\mu_g).
	\end{displaymath}
	At an atom, compactness is obtained from bounded point evaluation. At a
	nonatomic point $s$ in the support of $\mu_g$, one chooses shrinking
	half-open intervals $I_k$ containing $s$ such that
	\begin{displaymath}
		0<\mu_g(I_k)\longrightarrow0
	\end{displaymath}
	and considers the corresponding Bochner averages
	\begin{displaymath}
		A_{n,k}
		=
		\frac{1}{\mu_g(I_k)}
		\int_{I_k}w_n(r)\,d\mu_g(r).
	\end{displaymath}
	The support condition guarantees that the denominator is positive, while
	nonatomicity and continuity from above imply that the clock masses of the
	intervals tend to zero. Estimate \eqref{eq:time-control-Aubin-Lions}
	then compares the point value of the canonical representative with its
	local average.
	
	Hence the argument does not require $g$ to be injective, strictly
	increasing or absolutely continuous. The absolutely continuous,
	singular continuous and atomic parts of the clock are treated within one
	measure-theoretic framework. When $g(t)=t$, this framework becomes the
	usual Lebesgue-time proof. For a general derivator, it is $\mu_g$, rather
	than a classical reparametrization through $g$, that supplies the correct
	notions of temporal size, localization and compactness.
\end{rem}

\section{Consequences, time scales and model clocks}\label{sec:consequences}

This section first records convergence consequences that are useful in
nonlinear problems and recovers the classical and Gelfand-triple formulations.
It then identifies bounded time scales with canonical Stieltjes clocks and
specializes the compactness theorem to purely atomic and mixed clocks.  The
section ends with a separate discussion of the endpoint mechanisms and the
additional hypotheses that would be needed beyond the reflexive
$1<p_0,p_1<\infty$ setting.

\subsection{Functional-analytic consequences}\label{subsec:functional-consequences}

The state projection is automatically bounded.  If $c_{0B}$ denotes the
norm of the embedding $B_0\hookrightarrow B$, then
\begin{displaymath}
 \|\pi_0(u,z)\|_{L_g^{p_0}(B)}
 \leq c_{0B}\|u\|_{L_g^{p_0}(B_0)}
 \leq c_{0B}\|(u,z)\|_{\mathbb W_g^{p_0,p_1}}.
\end{displaymath}
The content of Theorem~\ref{thm:AubinLions} is therefore not continuity but
compactness: bounded subsets of the graph are mapped to relatively compact
subsets of $L_g^{p_0}(B)$.

\begin{cor}[Strong and almost-everywhere convergence]\label{cor:strong-ae}
Under the hypotheses of Theorem~\ref{thm:AubinLions}, every bounded sequence
$(u_n,z_n)$ in the evolution graph has a subsequence and a state
$u\in L_g^{p_0}([a,b);B)$ such that
\begin{displaymath}
 u_n\to u\quad\text{in }L_g^{p_0}([a,b);B)
\end{displaymath}
and, after a further subsequence,
\begin{displaymath}
 u_n(t)\to u(t)\quad\text{in }B
 \quad\text{for $g$-almost every }t\in[a,b).
\end{displaymath}
\end{cor}

\begin{proof}
The strong convergence follows from Theorem~\ref{thm:AubinLions}.  Choose
representatives of the state classes and then a further subsequence, not
relabelled, such that
\begin{displaymath}
 \sum_{n=1}^{\infty}
 \|u_{n+1}-u_n\|_{L_g^{p_0}(B)}<\infty.
\end{displaymath}
For $N\geq1$, Minkowski's inequality gives
\begin{displaymath}
 \left\|\sum_{n=1}^{N}
 \|u_{n+1}(\cdot)-u_n(\cdot)\|_B\right\|_{L_g^{p_0}}
 \leq
 \sum_{n=1}^{N}
 \|u_{n+1}-u_n\|_{L_g^{p_0}(B)}.
\end{displaymath}
The partial sums on the left are nonnegative and increase pointwise. Passing
to the limit by the monotone convergence theorem, applied to their
$p_0$-th powers, yields
\begin{displaymath}
 \left\|\sum_{n=1}^{\infty}
 \|u_{n+1}(\cdot)-u_n(\cdot)\|_B\right\|_{L_g^{p_0}}
 \leq
 \sum_{n=1}^{\infty}
 \|u_{n+1}-u_n\|_{L_g^{p_0}(B)}<\infty.
\end{displaymath}
Hence the series of pointwise increments is finite for $g$-almost every
$t$, and $(u_n(t))$ is Cauchy in $B$ there.  Denote its pointwise limit by
$v(t)$.  The subsequence converges to $v$ in measure, while it also converges
to $u$ in $L_g^{p_0}(B)$ and therefore in measure.  Uniqueness of limits in
measure yields $v=u$ almost everywhere.
\end{proof}

\begin{cor}[Passage through continuous nonlinearities]\label{cor:nonlinear-passage}
Assume the hypotheses of Theorem~\ref{thm:AubinLions}, and let
$F:B\to Z$ be continuous, where $Z$ is a Banach space.  If $(u_n,z_n)$ is
bounded in the evolution graph, then a subsequence satisfies
\begin{displaymath}
 u_n\to u\quad\text{in }L_g^{p_0}(B),
 \qquad
 F(u_n)\to F(u)\quad\text{in measure on }([a,b),\mu_g).
\end{displaymath}
If, for some $r\in[1,\infty)$, the family
$(\|F(u_n)\|_Z^r)$ is uniformly integrable, then
$F(u_n)\to F(u)$ strongly in $L_g^r(Z)$.
\end{cor}

\begin{proof}
Corollary~\ref{cor:strong-ae} gives, after extraction and a choice of
representatives, $u_n(t)\to u(t)$ in $B$ for $g$-almost every $t$.
Since each $u_n$ and $u$ is strongly measurable, it is the almost-everywhere
limit of simple functions with values in an essentially separable subset of
$B$. Composing those simple approximations with the continuous map $F$
shows that $F(u_n)$ and $F(u)$ are strongly measurable. Pointwise
continuity gives
$F(u_n(t))\to F(u(t))$ in $Z$ almost everywhere, hence in measure.

Assume now that $(\|F(u_n)\|_Z^r)$ is uniformly integrable. Fatou's lemma,
applied after the almost-everywhere extraction, shows that
$\|F(u)\|_Z^r$ is integrable. Moreover,
\begin{displaymath}
 \|F(u_n)-F(u)\|_Z^r
 \leq 2^{r-1}
 \bigl(\|F(u_n)\|_Z^r+\|F(u)\|_Z^r\bigr).
\end{displaymath}
The first family on the right is uniformly integrable by hypothesis, while
the singleton containing the integrable function $\|F(u)\|_Z^r$ is
uniformly integrable by the absolute continuity of its integral. The family
$(\|F(u_n)-F(u)\|_Z^r)$ is therefore uniformly integrable. Vitali's theorem
and convergence in measure imply
\begin{displaymath}
 \int_{[a,b)}\|F(u_n(t))-F(u(t))\|_Z^r\,d\mu_g(t)
 \longrightarrow0,
\end{displaymath}
which is exactly strong convergence in $L_g^r(Z)$.
\end{proof}

\begin{cor}[Classical Aubin--Lions compactness]\label{cor:classical-aubin-lions}
	Let $B_0$ and $B_1$ be reflexive Banach spaces,
	$B_0\Subset B\hookrightarrow B_1$, and $1<p_0,p_1<\infty$.  Then the
	space
	\begin{displaymath}
	 \mathcal W^{p_0,p_1}(a,b;B_0,B_1)
	 :=\{u\in L^{p_0}(a,b;B_0):
	       \partial_tu\in L^{p_1}(a,b;B_1)\},
	\end{displaymath}
	endowed with the graph norm
	\begin{displaymath}
	 \|u\|_{L^{p_0}(B_0)}+\|\partial_tu\|_{L^{p_1}(B_1)},
	\end{displaymath}
	is compactly embedded in $L^{p_0}(a,b;B)$.
\end{cor}

\begin{proof}
	Take $g(t)=t$.  Then $\mu_g$ is Lebesgue measure and the canonical
	representation of a graph pair is
	\begin{displaymath}
	 U(t)=x+\int_a^t z(r)\,dr.
	\end{displaymath}
	For every scalar test function $\varphi\in C_c^\infty(a,b)$, Bochner
	integration by parts gives
	\begin{displaymath}
	 \int_a^b U(t)\varphi'(t)\,dt
	 =-\int_a^b z(t)\varphi(t)\,dt
	 \quad\text{in }B_1.
	\end{displaymath}
	Thus $z$ is the distributional derivative of the state.  Conversely, if
	$u\in L^{p_0}(a,b;B_0)$ has a distributional derivative
	$z\in L^{p_1}(a,b;B_1)$, then both functions belong to
	$L^1(a,b;B_1)$ because the interval is finite.  The standard fundamental
	theorem for Sobolev--Bochner functions provides a representative
	$U(t)=x+\int_a^t z(r)\,dr$; see
	\cite[Sect.~7.1]{Roubicek2013}.  Hence $(u,z)$ belongs to the graph.

	The distributional derivative is unique, so the map
	$(u,z)\mapsto u$ identifies the graph isometrically with
	$\mathcal W^{p_0,p_1}(a,b;B_0,B_1)$ equipped with the graph norm stated
	above.  Applying
	Theorem~\ref{thm:AubinLions} proves the compact inclusion.  We retain the
	reflexivity and exponent assumptions of that theorem; sharper classical
	versions may be found in \cite{Aubin1963,Lions1969,Simon1987}.
\end{proof}

\begin{cor}[A Gelfand-triple form]\label{cor:gelfand-triple}
	Let $V\Subset H\hookrightarrow V'$ be a Gelfand triple, with $V$
	reflexive and $H$ a Hilbert space identified with its dual.  For every
	derivator $g$, the state projection
	\begin{displaymath}
	 \pi_0:\mathbb W_g^{2,2}([a,b];V,V')
	 \longrightarrow L^2_g([a,b);H)
	\end{displaymath}
	is compact.
\end{cor}

\begin{proof}
	The dual $V'$ is reflexive because $V$ is reflexive.  The assertion is
	Theorem~\ref{thm:AubinLions} with $(B_0,B,B_1)=(V,H,V')$ and
	$p_0=p_1=2$.
\end{proof}

\subsection{Time scales as Stieltjes clocks}\label{subsec:time-scales}

The relation between time-scale calculus and Stieltjes-type calculi has
already been established in the literature.  Slav\'\i k introduced the
canonical ceiling map of a time scale in order to embed dynamic equations
into the framework of generalized ordinary differential equations
\cite[Lemma~4]{Slavik2012}.  L\'opez Pouso and Rodr\'\i guez subsequently
showed that, after extending a function from the time scale by composition
with this map, Hilger differentiability is equivalent to Stieltjes
differentiability at points of the time scale
\cite[Theorem~3.1]{POUSO2015}.  Measure-theoretic representations of the
delta integral were developed in \cite{Guseinov2003,CabadaVivero2006}.

The purpose of this subsection is therefore not to claim the time-scale--
Stieltjes correspondence as new.  We formulate it in the present
Lebesgue--Stieltjes and Bochner-valued notation, identify the corresponding
evolution graphs, and use Theorem~\ref{thm:AubinLions} to derive a compactness
result on an arbitrary bounded time scale.  Let
$\mathbb T\subset\mathbb R$ be a nonempty closed set and let
$a,b\in\mathbb T$ with $a<b$.  We use the notation
\begin{displaymath}
	[a,b)_{\mathbb T}:=[a,b)\cap\mathbb T.
\end{displaymath}
For $t\in[a,b)_{\mathbb T}$, let
\begin{displaymath}
	\sigma_{\mathbb T}(t)
	:=\inf\{s\in\mathbb T:s>t\}
\end{displaymath}
be the forward jump operator and let
\begin{displaymath}
	\gamma_{\mathbb T}(t)
	:=\sigma_{\mathbb T}(t)-t
\end{displaymath}
be the graininess.  Since $b\in\mathbb T$, the set in the definition of
$\sigma_{\mathbb T}(t)$ is nonempty for every $t<b$.

\begin{pro}[A time scale as a Stieltjes clock]
	\label{prop:time-scale-clock}
	Define
	\begin{equation}
		\label{eq:time-scale-ceiling-clock}
		g_{\mathbb T}(r)
		:=\min\bigl(\mathbb T\cap[r,b]\bigr),
		\qquad r\in[a,b],
	\end{equation}
	and extend $g_{\mathbb T}$ constantly to the two components of
	$\mathbb R\setminus[a,b]$.  Then $g_{\mathbb T}$ is nondecreasing and
	left-continuous.  Moreover,
	\begin{displaymath}
		g_{\mathbb T}(t)=t
	\end{displaymath}
	for every $t\in[a,b]\cap\mathbb T$, and, at every right-scattered point
	$t\in[a,b)_{\mathbb T}$,
	\begin{equation}
		\label{eq:time-scale-graininess-jump}
		\Delta^+g_{\mathbb T}(t)
		=\sigma_{\mathbb T}(t)-t
		=\gamma_{\mathbb T}(t).
	\end{equation}
	
	Let $\mu_{\Delta}$ denote the Hilger measure on $[a,b)_{\mathbb T}$.  Then
	$\mu_{g_{\mathbb T}}$ is concentrated on $[a,b)_{\mathbb T}$ and
	\begin{equation}
		\label{eq:Hilger-Stieltjes-identification}
		\mu_{g_{\mathbb T}}(E)=\mu_{\Delta}(E)
	\end{equation}
	for every Borel set $E\subset[a,b)_{\mathbb T}$.  Consequently, for every
	Banach space $X$ and every Bochner integrable function
	$f:[a,b)_{\mathbb T}\to X$,
	\begin{equation}
		\label{eq:delta-Stieltjes-integral-identification}
		\int_{[a,b)_{\mathbb T}} f(t)\,\Delta t
		=
		\int_{[a,b)} f^{0}(t)\,\mathrm d\mu_{g_{\mathbb T}}(t),
	\end{equation}
	where $f^{0}$ denotes the extension of $f$ by zero outside the time
scale.
\end{pro}

\begin{proof}
	For each $r\in[a,b]$, the set $\mathbb T\cap[r,b]$ is nonempty and compact,
	so the minimum in \eqref{eq:time-scale-ceiling-clock} is well defined.  The
	map $g_{\mathbb T}$ is plainly nondecreasing.  To prove left continuity, let
	$r_n\uparrow r$.  Monotonicity gives the existence of the limit
	\begin{displaymath}
		\ell:=\lim_{n\to\infty}g_{\mathbb T}(r_n)
		\leq g_{\mathbb T}(r).
	\end{displaymath}
	Since $g_{\mathbb T}(r_n)\in\mathbb T$ and $\mathbb T$ is closed, we have
	$\ell\in\mathbb T$.  Moreover, $g_{\mathbb T}(r_n)\geq r_n$, and hence
	$\ell\geq r$.  Therefore
	\begin{displaymath}
		\ell\in\mathbb T\cap[r,b].
	\end{displaymath}
	By the minimality of $g_{\mathbb T}(r)$,
	\begin{displaymath}
		g_{\mathbb T}(r)\leq\ell.
	\end{displaymath}
	Thus $\ell=g_{\mathbb T}(r)$, proving left continuity.
	
	If $t\in\mathbb T$, then $t$ is the minimum of
	$\mathbb T\cap[t,b]$, and hence $g_{\mathbb T}(t)=t$.  If $t$ is
	right-scattered, then
	\begin{displaymath}
		g_{\mathbb T}(r)=\sigma_{\mathbb T}(t)
	\end{displaymath}
	for every $r\in(t,\sigma_{\mathbb T}(t)]$.  It follows that
	\begin{displaymath}
		g_{\mathbb T}(t^+)=\sigma_{\mathbb T}(t),
	\end{displaymath}
	which proves \eqref{eq:time-scale-graininess-jump}.
	
	The complement $[a,b]\setminus\mathbb T$ is an open subset of $\mathbb R$
	and is therefore a countable union of pairwise disjoint open intervals.  If
	$(\alpha,\beta)$ is one of these intervals, then
	\begin{displaymath}
		g_{\mathbb T}(r)=\beta
	\end{displaymath}
	for every $r\in(\alpha,\beta]$.  Hence the interior of every gap has zero
	$\mu_{g_{\mathbb T}}$-measure.  Therefore
	\begin{displaymath}
		\mu_{g_{\mathbb T}}([a,b)\setminus\mathbb T)=0.
	\end{displaymath}
	
	If $c,d\in\mathbb T$ and $a\leq c<d\leq b$, then
	\begin{displaymath}
		\begin{aligned}
			\mu_{g_{\mathbb T}}([c,d))
			&=g_{\mathbb T}(d)-g_{\mathbb T}(c)
			\\
			&=d-c.
		\end{aligned}
	\end{displaymath}
	Since the measure is concentrated on $\mathbb T$, this gives
	\begin{displaymath}
		\mu_{g_{\mathbb T}}([c,d)\cap\mathbb T)=d-c.
	\end{displaymath}
	Let
\begin{displaymath}
 \mathcal S_{\mathbb T}
 :=\{[c,d)\cap\mathbb T:a\leq c<d\leq b,\ c,d\in\mathbb T\}
 \cup\{\varnothing\}.
\end{displaymath}
This is a semiring that generates the relative Borel $\sigma$-algebra of
$[a,b)_{\mathbb T}$. The Hilger measure is characterized by
\begin{displaymath}
 \mu_{\Delta}([c,d)\cap\mathbb T)=d-c
\end{displaymath}
on $\mathcal S_{\mathbb T}$. Hence the restrictions of
$\mu_{g_{\mathbb T}}$ and $\mu_{\Delta}$ are two finite measures that
agree on a generating semiring. The uniqueness theorem for finite measure
extensions therefore yields
\eqref{eq:Hilger-Stieltjes-identification}; see
\cite{Guseinov2003,CabadaVivero2006}.
The integral identity
	\eqref{eq:delta-Stieltjes-integral-identification} follows from the equality
	of the measures and from the fact that the values of $f^{0}$ outside
	$\mathbb T$ are irrelevant.
\end{proof}

\begin{rem}[Relation with the existing time-scale correspondence]
\label{rem:known-time-scale-Stieltjes-correspondence}
The clock in \eqref{eq:time-scale-ceiling-clock} is the restriction to
$[a,b]$ of the ceiling map used by Slav\'\i k in the generalized-ODE
representation of dynamic equations on time scales
\cite[Lemma~4]{Slavik2012}.  If $f:\mathbb T\to X$ and
\begin{displaymath}
 f^{\sharp}(r):=f(g_{\mathbb T}(r)),
\end{displaymath}
then, in the scalar case, the equivalence between the Hilger derivative of
$f$ at a point of $\mathbb T$ and the Stieltjes derivative of
$f^{\sharp}$ with respect to $g_{\mathbb T}$ is the content of
\cite[Theorem~3.1]{POUSO2015}, under the natural left-continuity assumption
at right-scattered points.  This also explains the two local forms of the
derivative: at a right-scattered point one obtains the quotient
\begin{displaymath}
 \frac{f(\sigma_{\mathbb T}(t))-f(t)}
 {\sigma_{\mathbb T}(t)-t},
\end{displaymath}
whereas at a right-dense point one obtains the usual limit along the time
scale.

Proposition~\ref{prop:time-scale-clock} records the corresponding
measure-theoretic statement in the notation needed here.  In particular,
it identifies the Hilger measure with $\mu_{g_{\mathbb T}}$ and extends the
delta-integral identity to Bochner integrable functions.  The additional
step in Corollary~\ref{cor:Aubin-Lions-time-scale} is functional analytic:
the known scalar correspondence is lifted to the Stieltjes--Bochner
evolution graph, after which Theorem~\ref{thm:AubinLions} yields strong
compactness in the intermediate Banach space.
\end{rem}

For a Banach space $X$ and $1\leq p<\infty$, we write
\begin{displaymath}
	L_{\Delta}^p([a,b)_{\mathbb T};X)
	:=L^p([a,b)_{\mathbb T},\mu_{\Delta};X).
\end{displaymath}
Define the delta evolution graph
\begin{displaymath}
	\mathbb W_{\Delta}^{p_0,p_1}
	([a,b]_{\mathbb T};B_0,B_1)
\end{displaymath}
to be the set of pairs
\begin{displaymath}
	(u,z)\in
	L_{\Delta}^{p_0}([a,b)_{\mathbb T};B_0)
	\times
	L_{\Delta}^{p_1}([a,b)_{\mathbb T};B_1)
\end{displaymath}
for which there exist $x\in B_1$ and a representative
$U:[a,b]\cap\mathbb T\to B_1$ satisfying
\begin{displaymath}
 U(t)=i_{01}u(t)
 \quad\text{for $\mu_\Delta$-almost every }t\in[a,b)_{\mathbb T}
\end{displaymath}
and
\begin{equation}
	\label{eq:time-scale-evolution-graph}
	U(t)
	=x+
	\int_{[a,t)_{\mathbb T}}z(s)\,\Delta s,
	\qquad t\in[a,b]\cap\mathbb T.
\end{equation}
We endow this graph with the norm
\begin{displaymath}
	\|(u,z)\|_{\mathbb W_{\Delta}^{p_0,p_1}}
	:=
	\|u\|_{L_{\Delta}^{p_0}(B_0)}
	+
	\|z\|_{L_{\Delta}^{p_1}(B_1)}.
\end{displaymath}

\begin{cor}[Aubin--Lions compactness on a bounded time scale]
	\label{cor:Aubin-Lions-time-scale}
	Let $\mathbb T\subset\mathbb R$ be a time scale and let
	$a,b\in\mathbb T$ with $a<b$.  Suppose that $B_0$ and $B_1$ are reflexive
	Banach spaces,
	\begin{displaymath}
		B_0\Subset B\hookrightarrow B_1,
	\end{displaymath}
	and
	\begin{displaymath}
		1<p_0,p_1<\infty.
	\end{displaymath}
	Then the state projection
	\begin{displaymath}
		\pi_0:
		\mathbb W_{\Delta}^{p_0,p_1}
		([a,b]_{\mathbb T};B_0,B_1)
		\longrightarrow
		L_{\Delta}^{p_0}([a,b)_{\mathbb T};B),
		\qquad
		\pi_0(u,z)=u,
	\end{displaymath}
	is compact.
\end{cor}

\begin{proof}
	Let $g_{\mathbb T}$ be the clock defined in
	\eqref{eq:time-scale-ceiling-clock}.  For every Banach space $X$, extension by zero from
	$[a,b)_{\mathbb T}$ to $[a,b)$ gives an isometric
	identification
	\begin{displaymath}
		L_{\Delta}^p([a,b)_{\mathbb T};X)
		\cong
		L_{g_{\mathbb T}}^p([a,b);X),
	\end{displaymath}
	because of \eqref{eq:Hilger-Stieltjes-identification} and because
	$\mu_{g_{\mathbb T}}$ is concentrated on the time scale.
	
	Let $(u,z)$ belong to the delta evolution graph and let $U$ be the
	representative in \eqref{eq:time-scale-evolution-graph}.  Define
	\begin{displaymath}
		\widehat U(r):=U(g_{\mathbb T}(r)),
		\qquad r\in[a,b].
	\end{displaymath}
	Since $g_{\mathbb T}(t)=t$ on $\mathbb T$, the restriction of
	$\widehat U$ to the time scale is $U$.  Moreover, by
\eqref{eq:delta-Stieltjes-integral-identification},
\begin{displaymath}
 \widehat U(r)
 =x+
 \int_{[a,r)}z^{0}(s)\,
 \mathrm d\mu_{g_{\mathbb T}}(s),
 \qquad r\in[a,b],
\end{displaymath}
where $z^{0}$ is the extension of $z$ by zero outside the time scale.
Thus extension by zero maps the delta graph isometrically into the
Stieltjes evolution graph associated with $g_{\mathbb T}$.

Conversely, let $(v,w)$ belong to that Stieltjes evolution graph and let
$V:[a,b]\to B_1$ be its canonical representative. Because
$\mu_{g_{\mathbb T}}$ is concentrated on the time scale, restriction gives
well-defined classes
\begin{displaymath}
 u:=v|_{[a,b)_{\mathbb T}},
 \qquad
 z:=w|_{[a,b)_{\mathbb T}}
\end{displaymath}
in the corresponding delta Bochner spaces, without changing either norm.
For $t\in[a,b]\cap\mathbb T$, the integral identity in
Proposition~\ref{prop:time-scale-clock} gives
\begin{displaymath}
 V(t)
 =x+
 \int_{[a,t)}w(s)\,\mathrm d\mu_{g_{\mathbb T}}(s)
 =x+
 \int_{[a,t)_{\mathbb T}}z(s)\,\Delta s.
\end{displaymath}
Hence $(u,z)$ belongs to the delta evolution graph, with representative
$V|_{[a,b]\cap\mathbb T}$.  At the level of the corresponding Bochner
quotient spaces, extension and restriction induce mutually inverse isometries
on both graph components: their compositions are the identity on equivalence
classes, rather than a claim of pointwise identity for arbitrary representatives
off the support of the measure.  The two evolution graphs are therefore
isometrically identified as spaces of equivalence classes, and compactness of
the state projection follows from Theorem~\ref{thm:AubinLions}.
\end{proof}

\begin{rem}[Continuous, discrete and hybrid parts of a time scale]
	\label{rem:time-scale-decomposition}
	The preceding corollary contains the standard continuous and discrete cases.
	If $\mathbb T=[a,b]$, then $g_{\mathbb T}(t)=t$ and
	Corollary~\ref{cor:Aubin-Lions-time-scale} reduces to
	Corollary~\ref{cor:classical-aubin-lions}.  If every point of
	$[a,b)_{\mathbb T}$ is right-scattered, the Hilger measure is purely atomic
	and
	\begin{displaymath}
		\mu_{\Delta}(\{t\})=\gamma_{\mathbb T}(t).
	\end{displaymath}
	For a general time scale, the delta integral can be written as
	\begin{equation}
		\label{eq:time-scale-integral-decomposition}
		\int_{[a,b)_{\mathbb T}}\varphi(t)\,\Delta t
		=
		\int_{[a,b)\cap\mathbb T}\varphi(t)\,\mathrm dt
		+
		\sum_{\substack{t\in[a,b)_{\mathbb T}\\
				\sigma_{\mathbb T}(t)>t}}
		\gamma_{\mathbb T}(t)\varphi(t)
	\end{equation}
	for every nonnegative Borel function $\varphi$; the set of
	right-scattered points occurring in the sum is at most countable.  See
	\cite{Guseinov2003,CabadaVivero2006}.  Thus the same compactness statement controls both the diffuse
	Lebesgue contribution carried by the time scale and the atomic contribution
	generated by its right-scattered points.  No decomposition of the evolution
	argument into separate continuous and discrete proofs is required.
	
	For the canonical representative in
	\eqref{eq:time-scale-evolution-graph}, the derivative component agrees
	$\mu_{\Delta}$-almost everywhere with the strong delta derivative.  At a
	right-scattered point this identity reads
	\begin{displaymath}
		U^{\Delta}(t)
		=
		\frac{U(\sigma_{\mathbb T}(t))-U(t)}
		{\gamma_{\mathbb T}(t)}
		=z(t),
	\end{displaymath}
	whereas at a right-dense point it is the norm limit along the time scale.
	Consequently, Corollary~\ref{cor:Aubin-Lions-time-scale} is a genuine
	time-scale Aubin--Lions theorem, rather than only a reformulation of the
	purely atomic example below.
\end{rem}

\subsection{Atomic and mixed clocks}\label{subsec:model-clocks}

The next result is the exact graph-space counterpart of a purely atomic clock.  It is a compactness theorem for weighted sequences, not an asymptotic
statement about a family of time discretizations.  The weights determine both
the state norm and the discrete derivative norm.

\begin{cor}[Weighted discrete compactness]\label{cor:weighted-discrete}
	Let $B_0$ and $B_1$ be reflexive Banach spaces with
	$B_0\Subset B\hookrightarrow B_1$, let $1<p_0,p_1<\infty$, and let
	$\alpha=(\alpha_n)_{n\geq1}$ satisfy
	$\alpha_n>0$ and $\sum_{n=1}^{\infty}\alpha_n<\infty$.  Define
	\begin{displaymath}
		\begin{split}
		\mathcal X_\alpha(B_0,B_1):=\bigg\{x=(x_n)_{n\geq1}:{}&
		\sum_{n=1}^{\infty}\alpha_n\|x_n\|_{B_0}^{p_0}<\infty,\\
		&\sum_{n=1}^{\infty}\alpha_n^{1-p_1}
		\|x_{n+1}-x_n\|_{B_1}^{p_1}<\infty\bigg\}.
		\end{split}
	\end{displaymath}
	Endow this space with the norm
	\begin{displaymath}
		\|x\|_{\mathcal X_\alpha}
		:=\left(\sum_{n=1}^{\infty}\alpha_n
		\|x_n\|_{B_0}^{p_0}\right)^{1/p_0}
		+\left(\sum_{n=1}^{\infty}\alpha_n^{1-p_1}
		\|x_{n+1}-x_n\|_{B_1}^{p_1}\right)^{1/p_1}.
	\end{displaymath}
	Then $\mathcal X_\alpha(B_0,B_1)$ is a Banach space and the inclusion
	\begin{displaymath}
		\mathcal X_\alpha(B_0,B_1)
		\Subset \ell^{p_0}(\alpha;B),
		\qquad
		\|x\|_{\ell^{p_0}(\alpha;B)}^{p_0}
		=\sum_{n=1}^{\infty}\alpha_n\|x_n\|_B^{p_0},
	\end{displaymath}
	is compact.
\end{cor}

\begin{proof}
	Choose $a<s_1<s_2<\cdots<b$ with $s_n\uparrow b$ and define the
	left-continuous derivator
	\begin{displaymath}
		g(t)=\sum_{n=1}^{\infty}\alpha_n
		\mathbf 1_{(s_n,\infty)}(t).
	\end{displaymath}
	Thus $\mu_g=\sum_{n\geq1}\alpha_n\delta_{s_n}$ on $[a,b)$.
	Given $x\in\mathcal X_\alpha(B_0,B_1)$, set
	\begin{displaymath}
		u_x(t)=x_1\quad(a\leq t\leq s_1),
		\qquad
		u_x(t)=x_{n+1}\quad(s_n<t\leq s_{n+1}).
	\end{displaymath}
	The terminal value at $b$ is well defined in $B_1$.  Indeed, Hölder's
	inequality yields
	\begin{displaymath}
		\sum_{n=1}^{\infty}\|x_{n+1}-x_n\|_{B_1}
		\leq
		\left(\sum_{n=1}^{\infty}\alpha_n\right)^{1-1/p_1}
		\left(\sum_{n=1}^{\infty}\alpha_n^{1-p_1}
		\|x_{n+1}-x_n\|_{B_1}^{p_1}\right)^{1/p_1},
	\end{displaymath}
	so $(x_n)$ is Cauchy in $B_1$.  We complete the definition above by setting
	$u_x(b)=\lim_nx_n$ in $B_1$.  At the atom $s_n$ set
	\begin{displaymath}
		z_x(s_n)=\frac{x_{n+1}-x_n}{\alpha_n}.
	\end{displaymath}
	Set $z_x=0$ off the atoms.  The identity
	\begin{displaymath}
	 u_x(t)=x_1+\int_{[a,t)}z_x(r)\,d\mu_g(r),
	 \qquad t\in[a,b],
	\end{displaymath}
	follows by telescoping the jumps, with convergence in $B_1$ at $t=b$.
	Hence $(u_x,z_x)$ belongs to the evolution graph.  Direct calculation gives
	\begin{displaymath}
		\|u_x\|_{L^{p_0}_g(B_0)}^{p_0}
		=\sum_{n=1}^{\infty}\alpha_n\|x_n\|_{B_0}^{p_0},
		\qquad
		\|z_x\|_{L^{p_1}_g(B_1)}^{p_1}
		=\sum_{n=1}^{\infty}\alpha_n^{1-p_1}
		\|x_{n+1}-x_n\|_{B_1}^{p_1}.
	\end{displaymath}
	Conversely, let $(u,z)$ belong to the graph for this clock, and let $U$ be
its canonical representative. Theorem~\ref{t2.9} gives
$U\in AC_g([a,b];B_1)$. Since $g$ is constant between consecutive atoms,
Lemma~\ref{lem:ACg-Cg} shows that $U$ is constant there. Each singleton
$\{s_n\}$ has positive $\mu_g$-measure, so the value $u(s_n)\in B_0$ is
well defined for the Bochner class. Moreover,
\begin{displaymath}
 U(s_n)=i_{01}u(s_n),
\end{displaymath}
because $U=i_{01}u$ almost everywhere and equality must hold at every
positive-mass atom. Define
\begin{displaymath}
 x_n:=u(s_n)\in B_0.
\end{displaymath}
The atomic integral representation gives, in $B_1$,
\begin{displaymath}
 U(s_n^+)-U(s_n)=\alpha_nz(s_n),
 \qquad U(s_n^+)=U(s_{n+1})=i_{01}x_{n+1}.
\end{displaymath}
Consequently,
\begin{displaymath}
 z(s_n)=\frac{i_{01}(x_{n+1}-x_n)}{\alpha_n}.
\end{displaymath}
Identifying $B_0$ with its image in $B_1$, this is precisely the discrete
derivative used in the definition of $\mathcal X_\alpha(B_0,B_1)$. The
norm identities above now show that
$x\in\mathcal X_\alpha(B_0,B_1)$.
The map
	$x\mapsto(u_x,z_x)$ is therefore a linear bijection from
	$\mathcal X_\alpha(B_0,B_1)$ onto the graph associated with this atomic
	clock, and the two norm identities show that it is an isometry for the
	sum norms used here.  Proposition~\ref{prop:W-structure} proves that
	$\mathcal X_\alpha$ is complete.  Finally,
	\begin{displaymath}
	 \|u_x\|_{L_g^{p_0}(B)}^{p_0}
	 =\sum_{n=1}^\infty\alpha_n\|x_n\|_B^{p_0}.
	\end{displaymath}
	The compactness of the inclusion follows from
	Theorem~\ref{thm:AubinLions}.
\end{proof}

\begin{rem}[Finite atomic clocks and the last increment]
\label{rem:finite-atomic-clock}
If the preceding clock is replaced by
\begin{displaymath}
 \mu_g=\sum_{n=1}^N\alpha_n\delta_{s_n},
 \qquad s_1<\cdots<s_N<b,
\end{displaymath}
then $s_N$ is a terminal atom.  The state class records
$x_1,\ldots,x_N$, whereas the value
\begin{displaymath}
 x_{N+1}=x_N+\alpha_Nz(s_N)
\end{displaymath}
lies to the right of the last atom and is invisible to the
$L_g^{p_0}$ norm.  Accordingly, the finite analogue of
$\mathcal X_\alpha$ must either retain the last discrete derivative as a
separate variable or prescribe $x_{N+1}$ as endpoint data.  The graph
formulation does the former automatically.  The compactness theorem still
applies to the observed state $(x_1,\ldots,x_N)$; it does not falsely identify
the unobserved last increment from that state alone.
\end{rem}

\begin{exa}[A continuous--singular--atomic clock]\label{exa:mixed-clock}
	Let $C:[0,1]\to[0,1]$ be the Cantor staircase, extended constantly to
	$\mathbb R\setminus[0,1]$, let $\kappa\geq0$, and
	let $\{s_n\}\subset(0,1)$ be pairwise distinct with
	$\alpha_n>0$ and $\sum_n\alpha_n<\infty$.  The jump series converges
	uniformly, so the function
	\begin{displaymath}
		g(t)=t+\kappa C(t)+
		\sum_{n=1}^{\infty}\alpha_n\mathbf 1_{(s_n,\infty)}(t)
	\end{displaymath}
	is nondecreasing and left-continuous.  Its measure decomposition is
	\begin{displaymath}
		\mu_g=\mathcal L^1+\kappa\mu_C+
		\sum_{n=1}^{\infty}\alpha_n\delta_{s_n},
	\end{displaymath}
	where $\mu_C$ is the singular continuous Cantor measure.  Hence, for
	every nonnegative Borel function $\varphi$,
	\begin{displaymath}
		\int_{[0,1)}\varphi\,d\mu_g
		=\int_0^1\varphi(t)\,dt
		 +\kappa\int_{[0,1)}\varphi\,d\mu_C
		 +\sum_{n=1}^{\infty}\alpha_n\varphi(s_n).
	\end{displaymath}
	Theorem~\ref{thm:AubinLions} applies to the corresponding evolution graph.
	It gives strong compactness of the state when the derivative is controlled
	with respect to ordinary time, the Cantor measure and the atomic masses.
	The three components are handled in the same extraction.
\end{exa}

\begin{rem}[Plateaus of the clock]
	If $g$ is constant on an interval, every canonical $g$-absolutely
	continuous representative is constant there. The interiors of such
	intervals carry no $\mu_g$-mass. A right jump at an endpoint is an atomic
	contribution and must still be included whenever that endpoint belongs
	to $[a,b)$. Zero-mass plateau interiors create no obstruction to compactness.
	Thus Theorem~\ref{thm:AubinLions} does not require strict monotonicity
	of the clock.
\end{rem}

\subsection{Endpoint limitations and possible extensions}
\label{subsec:endpoint-mechanisms}

The following discussion concerns the mechanisms of the original proof of
Theorem~\ref{thm:AubinLions}. The derivative exponent controls its temporal
modulus, whereas the state exponent controls its weak compactness step.
The Dunford--Pettis type replacements below address those proof steps;
no necessity claim for the compactness conclusion is intended, and no
endpoint theorem is asserted here.

The restrictions on $p_0$ and $p_1$ enter the proof of
	Theorem~\ref{thm:AubinLions} at two different points. The condition
	$p_1>1$ provides a uniform modulus of continuity for the canonical
	representatives, whereas $p_0>1$, together with the reflexivity of $B_0$,
	provides weak compactness of the state component in its Bochner space.
	
	We first discuss the derivative exponent. If
	\begin{displaymath}
		1<p_1<\infty,
	\end{displaymath}
	H\"older's inequality gives
	\begin{displaymath}
		\int_{[r,t)}\|q_n(s)\|_{B_1}\,\mathrm{d}\mu_g(s)
		\leq
		\|q_n\|_{L_g^{p_1}(B_1)}
		\mu_g([r,t))^{1-1/p_1}.
	\end{displaymath}
	Since
	\begin{displaymath}
		\alpha:=1-\frac{1}{p_1}>0,
	\end{displaymath}
	the right-hand side tends to zero uniformly in $n$ whenever the local
	clock mass $\mu_g([r,t))$ tends to zero. This is precisely the mechanism
	used in \eqref{eq:time-control-Aubin-Lions}.
	
	At the endpoint $p_1=1$, the same computation yields only
	\begin{displaymath}
		\|W_n(t)-W_n(r)\|_{B_1}
		\leq
		\int_{[r,t)}\|q_n(s)\|_{B_1}\,\mathrm{d}\mu_g(s),
	\end{displaymath}
	and a uniform bound of the form
	\begin{displaymath}
		\sup_n\|q_n\|_{L_g^1(B_1)}<\infty
	\end{displaymath}
	does not force the last integral to be small on sets of small
	$\mu_g$-measure. This failure already occurs for ordinary Lebesgue time.
	For example, on $(0,1)$ the scalar functions
	\begin{displaymath}
		h_n(s):=n\,\mathbf{1}_{(0,1/n)}(s)
	\end{displaymath}
	satisfy
	\begin{displaymath}
		\|h_n\|_{L^1(0,1)}=1,
	\end{displaymath}
	but
	\begin{displaymath}
		\int_{(0,1/n)}|h_n(s)|\,\mathrm{d}s=1,
	\end{displaymath}
	even though the length of $(0,1/n)$ tends to zero. Thus boundedness in
	$L^1$ permits concentration of mass and does not produce a temporal
	modulus that is uniform over the family.
	
	The natural replacement is uniform integrability. Let $X$ be a Banach
	space and let
	\begin{displaymath}
		\mathcal F\subset L_g^1([a,b);X).
	\end{displaymath}
	The family $\mathcal F$ is called uniformly integrable if it is bounded
	in $L_g^1([a,b);X)$ and, for every $\varepsilon>0$, there exists
	$\delta>0$ such that
	\begin{displaymath}
		\mu_g(E)<\delta
	\end{displaymath}
	implies
	\begin{displaymath}
		\sup_{f\in\mathcal F}
		\int_E\|f(s)\|_X\,\mathrm{d}\mu_g(s)
		<\varepsilon
	\end{displaymath}
	for every Borel set $E\subset[a,b)$. Since $\mu_g$ is finite, this is
	equivalent to the tail condition
	\begin{displaymath}
		\lim_{R\to\infty}
		\sup_{f\in\mathcal F}
		\int_{\{\|f(s)\|_X>R\}}
		\|f(s)\|_X\,\mathrm{d}\mu_g(s)
		=0.
	\end{displaymath}
	For a family bounded in $L_g^1([a,b);X)$, define
	\begin{displaymath}
		\omega_{\mathcal F}(\eta)
		:=
		\sup_{f\in\mathcal F}
		\sup_{\substack{E\in\mathcal B([a,b))\\
				\mu_g(E)\leq\eta}}
		\int_E\|f(s)\|_X\,\mathrm{d}\mu_g(s),
	\end{displaymath}
	Then the family is uniformly integrable if and only if
	\begin{displaymath}
		\omega_{\mathcal F}(\eta)\longrightarrow0
		\qquad\text{as }\eta\downarrow0.
	\end{displaymath}
	Consequently, if the derivative family is uniformly integrable in
	$L_g^1(B_1)$, then the power modulus in
	\eqref{eq:time-control-Aubin-Lions} can be replaced by
	\begin{displaymath}
		\|W_n(t)-W_n(r)\|_{B_1}
		\leq
		\omega_{\mathcal F}\bigl(\mu_g([r,t))\bigr).
	\end{displaymath}
	This again tends to zero with the local clock mass. Such a hypothesis is
	therefore a natural candidate for treating the endpoint $p_1=1$.
	Alternative endpoint arguments may instead be based on a suitable
	translation criterion; compare \cite{Simon1987}.
	
	We next consider the state exponent. For
	\begin{displaymath}
		1<p_0<\infty,
	\end{displaymath}
	the reflexivity of $B_0$ implies the reflexivity of
	$L_g^{p_0}([a,b);B_0)$. Hence a bounded sequence of state components has
	a weakly convergent subsequence. At $p_0=1$, the space
	$L_g^1([a,b);B_0)$ is generally not reflexive, even if $B_0$ is
	reflexive, and boundedness alone no longer implies relative weak
	compactness. The concentration sequence displayed above already shows
	this obstruction in the scalar case.
	
	A Dunford--Pettis type hypothesis is a condition designed to replace
	reflexivity at this endpoint. In the scalar space $L^1$, the classical
	Dunford--Pettis theorem states that a bounded family is relatively weakly
	compact if and only if it is uniformly integrable; see
	\cite[Ch.~III, Sect.~2]{DiestelUhl1977}. For Bochner-valued spaces
	$L^1(\mu;X)$, uniform integrability is always necessary, but for a general
	Banach space $X$ it is not by itself a complete characterization of weak
	compactness. A precise vector-valued criterion is given in
	\cite[Theorem~2.1]{DiestelRuessSchachermayer1993}: in addition to uniform
	integrability, it requires an appropriate almost-everywhere weak
	compactness condition for convex combinations of sequences from the
	family.
	
	In the present theorem the range spaces $B_0$ and $B_1$ are assumed to be
	reflexive. In this reflexive setting, uniform integrability of a bounded
	family in $L_g^1(B_i)$ is a standard sufficient condition for its relative
	weak compactness. Thus, at $p_0=1$, a natural Dunford--Pettis type
	replacement for the reflexive $L^{p_0}$ argument would be to assume that
	the state family $\mathcal U\subset L_g^1(B_0)$ satisfies
	\begin{displaymath}
		\sup_{u\in\mathcal U}\|u\|_{L_g^1(B_0)}<\infty
	\end{displaymath}
	and
	\begin{displaymath}
		\lim_{\delta\downarrow0}
		\sup_{u\in\mathcal U}
		\sup_{\substack{E\in\mathcal B([a,b))\\
				\mu_g(E)\leq\delta}}
		\int_E\|u(s)\|_{B_0}\,\mathrm{d}\mu_g(s)
		=0.
	\end{displaymath}
	Under such a condition, one can recover the weakly convergent subsequence
	that boundedness in $L_g^{p_0}(B_0)$ provides automatically when
	$p_0>1$.
	
	The same distinction applies to the derivative component. If $p_1>1$ and
	$B_1$ is reflexive, boundedness in $L_g^{p_1}(B_1)$ gives weak compactness
	automatically. If $p_1=1$, uniform integrability of the derivatives would
	serve two purposes: it would provide the vanishing temporal modulus above
	and, because $B_1$ is reflexive, it would also yield relative weak
	compactness in $L_g^1(B_1)$.
	
	These observations indicate plausible hypotheses for endpoint variants,
	but they do not constitute an endpoint theorem. A complete result would
	still have to verify that the graph is stable under the resulting weak
	convergence and that the final compactness argument remains valid with the
	modified modulus. The assumptions
	\begin{displaymath}
		1<p_0,p_1<\infty
	\end{displaymath}
	and the reflexivity of $B_0$ and $B_1$ avoid these additional issues and
	lead to a transparent statement. No claim that these assumptions are
	optimal is intended.

\section{Conclusions}\label{sec:conclusions}

The fundamental theorem obtained here places Stieltjes absolute continuity
on the same footing as ordinary Banach-valued absolute continuity, provided
the range has the Radon--Nikod\'ym property.  The vector-measure proof also
shows that the density of the variation is exactly the norm of the strong
$g$-derivative.  This identity is useful when an a priori estimate is first
available at the level of total variation.

The compactness theorem requires one structural adjustment.  A terminal
atom may hide the final increment of a left-continuous trajectory from its
Bochner state class; the natural object is therefore the graph of
state--derivative pairs.  Its state projection is compact when
$B_0\Subset B\hookrightarrow B_1$, $1<p_0,p_1<\infty$, and the two outer
spaces are reflexive.  The proof reflects the measure itself: atoms are dealt
with by bounded evaluation, nonatomic points by local averages, and
Ehrling's inequality joins the resulting $B_1$ convergence to the required
$B$ convergence.

Ordinary time, arbitrary bounded time scales, summable weighted atomic
clocks, and clocks combining Lebesgue, Cantor and atomic measures all fit the
same statement.  The resulting strong convergence is precisely the form
needed to pass to many continuous nonlinear terms in weak evolution
problems, as recorded in Corollary~\ref{cor:nonlinear-passage}.  The
measure-theoretic formulation treats all components of the clock in one
extraction and avoids an inverse-time construction, which would be
unavailable in the presence of plateaus or jumps.

In the original proof, an $L_g^1$ bound on the derivative alone does not
yield a vanishing temporal modulus; uniform integrability is one sufficient
replacement. At $p_0=1$, uniform integrability of the states is a sufficient
way to recover that proof's weak compactness step under the stated reflexivity
assumptions. These observations concern those proof mechanisms and do not
assert that such additional hypotheses are necessary for compactness.
Endpoint and generalized formulations are not developed here.

\section*{CRediT authorship contribution statement}

Francisco J. Fern\'andez: Conceptualization, Methodology, Formal analysis,
Investigation, Writing--original draft, Writing--review and editing.

\section*{Funding}

The author was supported by the Xunta de Galicia through the project
``Consolidaci\'on e Estruturaci\'on 2023 GRC GI-1561---Ecuaci\'ons
diferenciais non lineais (EDNL).''

\section*{Declaration of competing interest}

The author declares that he has no known competing financial interests
or personal relationships that could have appeared to influence the work
reported in this paper.

\section*{Data availability}

No data were used or generated for the research described in this article.

\section*{Declaration of generative AI and AI-assisted technologies in the
writing process}

During the preparation of this manuscript, the author used OpenAI ChatGPT to
improve readability and language and to assist with manuscript organization.
After using this tool, the author reviewed and edited the content as needed,
independently checked the mathematical arguments and references, and takes
full responsibility for the content of the article.

\end{document}